\documentclass[12pt,final]{amsart}
\usepackage[asymmetric, centering]{geometry}
\usepackage{graphicx}
\usepackage[colorlinks=true, citecolor=blue, linkcolor=blue, urlcolor=blue]{hyperref}
\usepackage{mathrsfs}
\usepackage{amsfonts}

\usepackage{amsmath}
\usepackage{amssymb}
\usepackage{fancyhdr}
\usepackage{bm}
\usepackage{ifdraft}
\ifoptionfinal{
\usepackage[disable]{todonotes}
}{
\usepackage[norefs, nocites]{refcheck}

\usepackage[notref, notcite]{showkeys}
\usepackage[bordercolor=white, color=white]{todonotes}
}

\makeatletter\providecommand\@dotsep{5}\def\listtodoname{List of Todos}\def\listoftodos{\hypersetup{linkcolor=black}\@starttoc{tdo}\listtodoname\hypersetup{linkcolor=blue}}\makeatother

\newtheorem{lemma}{Lemma}
\newtheorem{proposition}{Proposition}
\newtheorem{theorem}{Theorem}

\newtheorem{definition}{Definition}
\theoremstyle{remark}

\newtheorem{remark}{Remark}

\def\C{\mathbb C}
\def\R{\mathbb R}

\def\N{\mathbb N}

\def\U{\mathbf U}
\def\P{\mathbf P}
\def\u{\mathbf u}
\DeclareMathOperator{\id}{id}
\DeclareMathOperator{\su}{\mathfrak{su}}
\DeclareMathOperator{\SU}{SU}
\DeclareMathOperator{\GL}{GL}
\DeclareMathOperator{\Ad}{Ad}
\DeclareMathOperator{\vol}{vol}

\DeclareMathOperator{\YM}{YM}
\DeclareMathOperator{\singsupp}{singsupp}

\def \ba {\begin {eqnarray*} }
\def \ea {\end {eqnarray*} }
\def \beq {\begin {eqnarray}}
\def \eeq {\end {eqnarray}}

\def\g{\mathfrak g}
\def\gc{{\mathfrak g}_{\mathbb{C}}}
\def\h{\mathfrak h}
\def\t{\mathfrak t}

\def\p{\partial}
\DeclareMathOperator{\supp}{supp}
\newcommand{\pair}[1]{\left\langle #1 \right\rangle}

\DeclareMathOperator{\WF}{WF}

\title{Inverse problem for the Yang--Mills equations}

\author[X. Chen]{Xi Chen}
\address{ Department of Pure Mathematics and Mathematical Statistics, University of Cambridge, Cambridge CB3 0WB, UK. {\it E-mail address: \bf \tt xi.chen@dpmms.cam.ac.uk}}

\author[M. Lassas]{Matti Lassas}
\address{Department of Mathematics and Statistics, University of Helsinki. {\it E-mail address: \bf \tt Matti.Lassas@helsinki.fi}}

\author[L. Oksanen]{Lauri Oksanen}
\address{Department of Mathematics, University College London, Gower Street, London WC1E 6BT, UK. {\it E-mail address: \bf \tt l.oksanen@ucl.ac.uk}}

\author[G.P. Paternain]{Gabriel P. Paternain}
\address{ Department of Pure Mathematics and Mathematical Statistics, University of Cambridge, Cambridge CB3 0WB, UK. {\it E-mail address: \bf \tt g.p.paternain@dpmms.cam.ac.uk}}

\begin{document}
\begin{abstract} We show that a connection can be recovered up to gauge from source-to-solution type data associated with the Yang--Mills equations in Minkowski space $\R^{1+3}$. Our proof analyzes the principal symbols of waves generated by suitable nonlinear  interactions and reduces the inversion to a broken non-abelian light ray transform. The principal symbol analysis of the interaction is based on a delicate calculation that involves the structure of the Lie algebra under consideration and the final result holds for any compact Lie group.

\end{abstract}
\maketitle


\section{Introduction}

The purpose of this paper is to solve an inverse problem associated with Yang--Mills theories in Minkowski space $\R^{1+3}$. The objective is the recovery of the gauge field $A$ on a causal domain where waves can propagate and return, given data on a small observation set inside the domain.

The starting point of Yang--Mills theories is a compact Lie group $G$ with Lie algebra $\mathfrak{g}$. Without loss of generality, we shall think of $G$ as a matrix Lie group and hence $\mathfrak{g}$ will be a matrix Lie algebra.
We assume also that $G$ is connected and endowed with a bi-invariant metric, or equivalently, an inner product on $\mathfrak{g}$ invariant under the adjoint action.

In their most general formulation, Yang--Mills theories take place in the adjoint bundle of a principal bundle with structure group $G$ over space-time. Since our region of interest in space-time will be a contractible set $M\subset \R^{1+3}$, we might as well assume from the start that we are working with the trivial adjoint bundle $M\times \mathfrak{g}$.
The main object of the theory is a gauge field $A$, also known as Yang--Mills potential. In geometric language this is simply a connection $A\in C^{\infty}(M;T^*M\otimes\mathfrak{g})=\Omega^{1}(M;\mathfrak{g})$, that is, a smooth $\mathfrak{g}$-valued 1-form. In general, we denote the set of $\mathfrak{g}$-valued forms of degree $k$ by $\Omega^{k} = \Omega^{k}(M;\mathfrak{g})$.

There is a natural pairing
$[\cdot, \cdot] :
\Omega^{p}\otimes\Omega^{q}\to \Omega^{p+q}$
given in our situation as
\[[\omega,\eta]=\omega\wedge\eta-(-1)^{pq}\eta\wedge\omega,\]
where the wedge product of $\mathfrak{g}$-valued forms is understood using matrix multiplication in $\mathfrak{g}$. Using the pairing we define a covariant derivative
\[d_{A}:\Omega^{k}(M;\mathfrak{g})\to \Omega^{k+1}(M;\mathfrak{g}), \qquad d_{A}\omega=d\omega+[A,\omega].\]
Given a gauge field $A$, we can associate to it, its field strength or {\it curvature}. This is defined as
\[F_{A}:=dA+\frac{1}{2}[A,A]=dA+A\wedge A \in \Omega^{2}(M;\mathfrak{g})\]
and it always satisfies the Bianchi identity $d_{A}F_{A}=0$. Moreover, $d_A^2 \omega = [F_A, \omega]$ for any $\omega \in \Omega^k$.

\subsection{Yang--Mills equations}

The Yang--Mills equations arise as the Euler--Lagrange equations for the Yang--Mills action functional which we now recall. The inner product in $\mathfrak{g}$ naturally induces a pairing $\langle \cdot,\cdot\rangle_{\text{Ad}}$
\[\Omega^{p}(M,\mathfrak{g})\times\Omega^{q}(M,\mathfrak{g})\to \Omega^{p+q}(M).\]
If  $\star$ denotes the Hodge star operator of the Minkowski metric, the Yang-Mills functional is given by
\[S_{\YM}(A):=\frac{1}{2}\int_{M}\langle F_{A},\star F_{A}\rangle_{\text{Ad}}.\]
If $G$ is a subgroup of the unitary group, we may take as adjoint invariant inner product $-\text{trace}(XY)$, where $X,Y$ are matrices in $\mathfrak{g}$, and thus $S_{\YM}(A)$ may also be written as a constant multiple of
\[\int_{M}\text{trace}((F_{A})_{\alpha \beta}F_{A}^{\alpha \beta})\,d\text{vol},\]
as is frequently found in the physics literature. From this functional one easily derives
the Yang--Mills equations:
\begin{equation}
d_{A}^*F_{A}=0, \label{eq:1}
\end{equation}
where $d_{A}^*$ is the formal adjoint of $d_A$ and given by
 \begin{eqnarray*}d_{A}^*:\Omega^{k}(M;\mathfrak{g})\to \Omega^{k-1}(M;\mathfrak{g}),  \qquad d_{A}^*=\star d_{A}\star.\end{eqnarray*}
(In general for a Lorentzian space-time of dimension $m$, the formal ajoint acting on $k$-forms has the expression $d_{A}^*=(-1)^{m+km}\star d_{A}\star$.)

The Yang--Mills equations are gauge invariant in the sense that if
two connections $A$ and $B$ are gauge equivalent and if
$A$ satisfies (\ref{eq:1}) then also $B$ satisfies $d_{B}^*F_{B}=0$.
The connections $A$ and $B$ being gauge equivalent means that
there is a section $\U \in C^\infty(M; G)$ such that
    \begin{align}\label{gauge_equiv}
B = \U^{-1} d \U + \U^{-1} A \U.
    \end{align}
This property can be easily deduced from the fact that the action $S_{\YM}$ is gauge invariant.

\subsection{Main result}
We will consider an inverse problem for the Yang--Mills equations in the causal diamond
$$
\mathbb D
= \{ (t,x) \in \R^{1+3} : |x| \le t + 1,\ |x| \le 1 - t \}.
$$
For a fixed $0 < \epsilon_0 < 1$, the data will be given on the subset
    \begin{align}\label{def_mho}
\mho = \{(t,x) : \text{$(t,x)$ is in the interior of $\mathbb D$ and $|x| < \epsilon_0$} \}.
    \end{align}
We we say that $A \in \Omega^{1}(\mathbb D;\mathfrak{g})$ is a background connection if
it satisfies the Yang--Mills equations (\ref{eq:1}) in $\mathbb D$.
Due to the gauge invariance, the determination of a background connection on $\mathbb D$ is considered only up to the action of the following pointed gauge group
    \begin{align*}
G^0(\mathbb D,p) =  \{\U \in C^\infty(\mathbb D; G) : \U(p) = \id \},
    \end{align*}
where $p = (-1, 0) \in \overline\mho$. The reason for considering the pointed gauge group instead of the full
gauge group
 \begin{align*}
G(\mathbb D) =  C^\infty(\mathbb D; G),
    \end{align*}
is technical in nature as we shall explain below, see discussion after Lemma \ref{lem_mod_data}. Both gauge groups are clearly related by $G(\mathbb D)/G^{0}(\mathbb D,p)=G$.

For $A, B \in C^k(\mathbb D;T^*\mathbb D\otimes\mathfrak{g})$, with $k \in \N$,
we say that $A \sim B$ in $\mathbb D$ if there is $\U \in G^0(\mathbb D,p)$ such that (\ref{gauge_equiv}) holds in $\mathbb D$.
Moreover, we write
    \begin{align*}
\p^- \mathbb D = \{ (t,x) \in \mathbb D : |x| = t + 1 \}
    \end{align*}
and say that $A \sim B$ near $\p^- \mathbb D$ if there are  $\U \in G^0(\mathbb D,p)$ and a neighbourhood $\mathcal U \subset \mathbb D$ of $\p^- \mathbb D$
such that (\ref{gauge_equiv}) holds in $\mathcal U \cap \mathbb D$.
The sets $\mathbb D$, $\mho$ and $\p^- \mathbb D$ are visualized in Figure~\ref{fig_D}.

\begin{figure}
\centering
\includegraphics[width=0.5\textwidth,trim={3cm 5cm 6cm 2cm},clip]{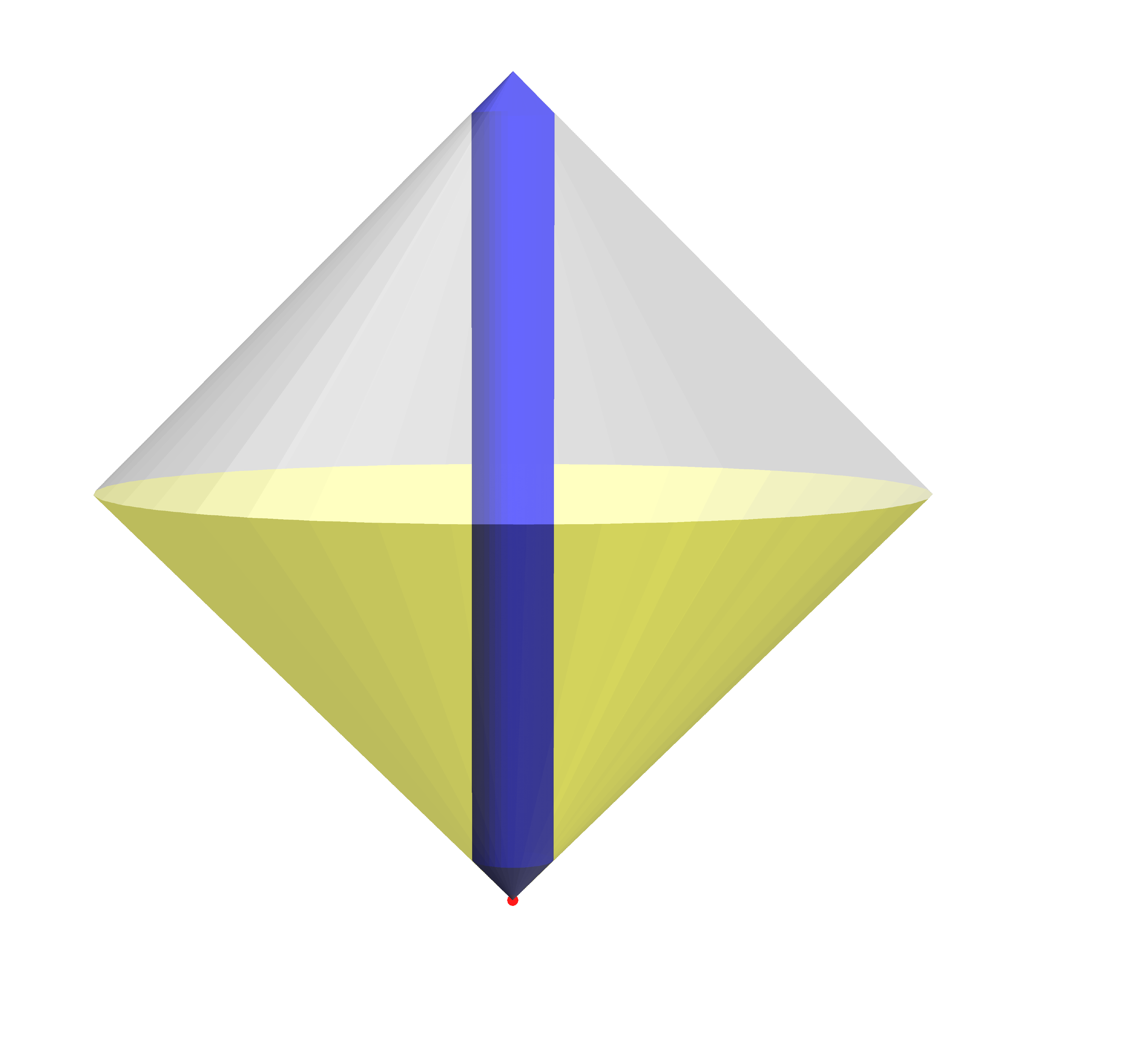}
\caption{The set $\mho$ (in blue) inside the diamond $\mathbb D$ in the $1+2$ dimensional case. The part $\p^- \mathbb D$ of the boundary of $\mathbb D$ is shaded in yellow. The point $p$ is drawn in red.}
\label{fig_D}
\end{figure}

We let $A$ be a background connection, and consider the data set
    \begin{align*}
\mathcal D_A = \{ V|_\mho :\ &\text{$V \in C^3(\mathbb D; T^* \mathbb D \otimes \mathfrak{g})$ satisfies $d_{V}^*F_{V}=0$ in $\mathbb D \setminus \mho$}
\\& \nonumber
\text{and $V \sim A$ near $\p^- \mathbb D$} \}.
    \end{align*}
Let us remark that we could consider the source-to-solution map given in Proposition~\ref{prop_source_to_sol} instead of the more abstract data set $\mathcal D_A$. We prefer to formulate our main result using $\mathcal D_A$ since the definition of the source-to-solution map is technical, requiring suitable gauge fixing among other things. In fact, it is precisely in the proof of Proposition ~\ref{prop_source_to_sol}
that the pointed gauge group is needed.
Nevertheless, intuitively, it is helpful to think of the data set as that produced by an observer creating sources $J$ supported in $\mho$ and observing solutions $V$ to $d_{V}^*F_{V}=J$
in $\mho$.

The data set $\mathcal D_A$ could also be reformulated in terms of the pairs $(J, V|_\mho)$ satisfying $d_{V}^*F_{V}=J$, with $J$ supported in $\mho$. This formulation, while being somewhat redundant as $J = d_{V}^*F_{V}$ can be computed given $V|_\mho$, suggests viewing $\mathcal D_A$ informally as the graph of the map taking $J$ to $V|_\mho$. However,
we reiterate that defining such map requires care.
In addition to gauge fixing, we need to take into account the compatibility condition $d_{V}^*J=0$ that every source must satisfy, see Lemma \ref{lem_div_J}.
Our abstract formulation of the data set $\mathcal D_{A}$ bypasses these problems while incorporating the natural gauge invariance of the theory.

We are now ready to formulate our main result.

\begin{theorem}\label{th_main}
Suppose that $A, B \in \Omega^{1}(\mathbb D;\mathfrak{g})$ solve (\ref{eq:1}) in $\mathbb D$. Then
$\mathcal D_A = \mathcal D_B$ if and only if $A \sim B$ in $\mathbb D$.
\end{theorem}

Clearly if $A \sim B$ in $\mathbb D$ then $\mathcal D_A = \mathcal D_B$.
The non-trivial content of the theorem is the opposite implication.
It follows from Proposition \ref{prop_abstract_uniq_pointed} in Appendix \ref{appendix_energy} that if $A$ and $B$ are as in the theorem, then
$A \sim B$ in $\mathbb D$ if and only if $A \sim B$ near $\p^- \mathbb D$.

\subsection{Outline of the proof of Theorem \ref{th_main}} The objective is to reduce the proof of the theorem to
an inversion result for a broken non-abelian light ray transform as in \cite{CLOP}.  The broken light ray transform that arises in this paper is that related to the adjoint representation given the natural habitat of the Yang--Mills theories. In \cite{CLOP} we studied the broken light ray transform associated with the fundamental representation, so our first task is to relate the two.

To go from the data set $\mathcal D_{A}$ to the broken non-abelian light transform we follow the template laid out in \cite{CLOP} where a considerably simpler wave equation with cubic non-linearity was studied.
The first step is then to process the abstract data set and convert it into a manageable source-to-solution map
and this already brings the question of gauge fixing to the forefront.
The construction of source-to-solution map uses two types of gauges: the temporal gauge and the relative Lorenz gauge. The {\em temporal gauge} is easy to implement as it involves solving a linear matrix ODE to make the time component of a Yang--Mills potential $A$ to vanish, that is, $A_0=0$. This gauge is particularly suited to prove uniqueness results, cf. Proposition \ref{prop_tempg_uniq} below.

It is important to remark that uniqueness does really depend on the shape of the set where the connections satisfy the Yang--Mills equations. The causal diamond $\mathbb D$ has the special feature that perturbations cannot
propagate in it through the top boundary $|x|=1-t$, whereas the bottom
boundary is under control due to the assumed gauge equivalence near $\p^- \mathbb D$. In particular, even if a background connection $A$ satisfies the Yang--Mills equations on a larger set than $\mathbb D$, we do not expect to be able to recover it outside $\mathbb D$ given data on $\mho$. Moreover, it does not appear to be possible to prove Theorem \ref{th_main} using presently known unique continuation results, as discussed in more detail below.

A connection $V$ is said to be in {\it relative Lorenz gauge} with respect to the background $A$ if $d_{A}^* V = d_{A}^* A$.
The advantage of this gauge is that if $A$ satisfies Yang--Mills $d_{A}^*F_{A}=0$, and $d_{V}^*F_{V}=J$, then the difference $W=V-A$
satisfies a semilinear wave equation where the leading part is given by the connection wave operator
$\Box_{A}=d_{A}d^{*}_{A}+d_{A}^*d_{A}$, cf. (\ref{eq_YM_relative_L}). This is very helpful for solving the foward problem and for the microlocal analysis used to extract information from the source-to-solution map.

Following \cite{CLOP}, the idea is to consider the non-linear interaction of three singular waves produced by sources which are conormal distributions. We carefully track the principal symbol produced by the non-linear interaction and extract from that the non-abelian broken light ray transform. This requires a delicate calculation
unlike anything in the previous literature, in which the structure of the Lie algebra $\mathfrak{g}$ comes into consideration.  This is the technical core of the proof, and perhaps one of the most innovative aspects of the paper. After this computation, contained in Section~\ref{subsec : principal symbol}, there is one further hurdle to overcome: to use the source-to-solution map
we must revert back to the temporal gauge and check that no information is lost in the process.

\subsection{Discussion and comparison with previous literature} It is tempting to think that a result like Theorem \ref{th_main} can be obtained from a unique continuation principle. It must be stressed that unique continuation for linear wave equations with time-dependent coefficients is simply false as there are counterexamples \cite{Alinhac1983}. Although the difference of two solutions to the Yang--Mills equations in the Lorenz gauge satisfies a linear wave equation (with coefficients depending on both the solutions),
due to unique continuation failing, our inverse problem is not ``immediately solvable" and hence a different approach is needed.
We mention that an inverse problem for Yang--Mills connections on a Riemannian manifold was studied in \cite{Cekic2017}. The proofs there are based on 
unique continuation for elliptic systems, however, the elliptic case is very different from the hyperbolic one.

This paper sits firmly within the program, initiated in \cite{CLOP}, that is motivated by the
Yang--Mills--Higgs system. In addition to the Yang--Mills potential $A$, a Higgs field $\Phi\in C^{\infty}(M,\mathfrak{g})$ is present in this system. The equations for the pair of fields $(A,\Phi)$ are given by
\begin{align}
&d_{A}^*F_{A}+[\Phi,d_{A}\Phi]=0;\label{eq:1b}\\
&d_{A}^*d_{A}\Phi +V'(|\Phi|^{2})\Phi=0,\label{eq:2}
\end{align}
where $V'$ is the derivative of a smooth function $V: [0,\infty)\to\mathbb{R}$. More generally, we can consider these equations when $\Phi$ is a section of an associated bundle determined by a given representation of $G$.
The focus of \cite{CLOP}
was the recovery of $A$ via the second equation \eqref{eq:2}, when $V$ is assumed to be a quadratic potential
(the most popular choice in Yang--Mills--Higgs theories): this turns \eqref{eq:2} into a wave equation with a cubic non-linearity. The present paper focuses on the first equation \eqref{eq:1b}; more precisely in the pure
Yang--Mills case where $\Phi=0$.
There are two substantial differences between \cite{CLOP} and the present paper. First, when $A$ is fixed, the second equation (\ref{eq:2}) is no more gauge invariant, and hence the construction of source-to-solution map in \cite{CLOP} does not require gauge fixing.
Second, the quadratic potential $V$ leads to particularly simple non-linear structure in \cite{CLOP}, and the resulting analysis of principal symbols is much more straightforward than in the present paper.

As already mentioned above, we consider the non-linear interactions of three singular waves.
Interaction of singular waves has been studied outside the context of inverse problems. In particular, the wave front set of a triple cross-derivative has been studied in the case of the $1+2$-dimensional Minkowski space by Rauch and Reed \cite{Rauch1982}. The references \cite{Bony,MR1623392,Melrose1985,MR918380,Barreto2018}
have results of similar nature. 
The use of non-linear interactions in the context of inverse problems was initiated in \cite{KLU},
where the wave front set resulting from the interaction of four singular waves was studied.
The same approach was used for the Einstein equations in \cite{KLOU}, and subsequently in \cite{Lassas2017, Uhlmann2018}, in some ways the closest previous results to ours.
For a review of this approach, see \cite{Matti}.
We observed in our above mentioned work \cite{CLOP}
that it is sufficient to consider interactions of three singular waves, simplifying the analysis. Three-fold interactions are used in the present paper. 

Non-linearities allow solving inverse problems that are open for the corresponding linearized equations.
In particular, the inverse problem for the linearized Yang--Mills equation, see e.g. (\ref{eqn : 1-fold interaction}) below (where some lower order terms are discarded),
is open.
The only known results are in the case $G=U(1)$, see \cite{MR3104931, Feizmohammadi2019}, and these results impose convexity assumptions not satisfied by the geometric setting of Figure \ref{fig_D}.
The same is true for recovery zeroth order terms, solved with and without convexity assumptions for certain scalar linear \cite{MR1004174} and non-linear wave equations \cite{Feizmohammadi2019a}, respectively.

We mention that non-linear interactions
have also
been used to recover non-linear terms for scalar wave equations \cite{LUW},
scalar elliptic equations \cite{Feiz,LLLS1},
and scalar real principal type equations \cite{Oksanen2020}.
In these four works, non-linear terms do not contain any derivatives, contrary to the Einstein and Yang--Mills equations. Non-linear interactions involving derivatives have also been studied in the context of scalar wave equations \cite{MR3995093} and elastodynamics \cite{MR4055177}.
In addition, inverse problems have been studied for various non-linear equations using methods originally developed in the context of linear elliptic equations.  In particular,
the method of complex geometrical optics originating from \cite{MR873380}, and importantly extended by \cite{MR970610,MR2299741}, was first applied 
to an inverse coefficient determination problem for a non-linear parabolic equation \cite{MR1233645} and subsequently to several other inverse problems \cite{Assylbekov1,MR3964222,MR1311909,MR1295934,MR1929283,MR3023384,
MR1465069}. 

There are numerous analogies between the problem studied here and that of the Einstein equations
considered in \cite{KLOU}. For starters, both problems have gauges: in the Einstein case the gauge group is the diffeomorphism group. The role of the relative Lorenz gauge is played by wave coordinates and one could also say that
the Fermi coordinates used in \cite{KLOU} are the analogue of the temporal gauge. Both problems have a compatibility condition for the sources: the Einstein tensor has zero divergence and Yang--Mills has
$d_{A}^*d_{A}^*F_{A}=0$.

However, there are important differences and we want to stress those, since they
are essential in resolving the inverse problem in the different contexts.
After suitable gauge fixing and linearization, both the Einstein and Yang--Mills equations reduce to a linear wave equation.
The unknown Lorentzian metric appears in the leading order terms of the equation in the former case while the background gauge field $A$ features at the subprincipal level in the latter case.
The Lorentzian metric affects the Lagrangian geometry of the parametrix for the wave equation but the effect of $A$ is visible only in the principal symbol of the parametrix.
Thus the need for a symbol calculation in the present paper that takes into consideration the structure of the Lie algebra $\mathfrak{g}$.
Finally, the two inverse problems reduce to very different purely geometric problems. In our case, we read the broken non-abelian light ray transform from certain principal symbols, whereas in the Einstein case, the so-called light observation sets are obtained by analysing the wave front sets of suitable solutions, see \cite{KLU,MR4032181} for the corresponding geometric problem.

\subsection{Outline of the paper} Section \ref{sec_pt} introduces parallel transport in both the principal and the adjoint representation and reduces Theorem \ref{th_main} to inversion of the broken non-abelian light ray transform via \cite[Proposition 2]{CLOP} in the case that $G$ has finite centre. Section \ref{sec_ym}
discusses the Yang--Mills equations with a source.
Section \ref{sec_gauges} introduces the relative Lorenz gauge and the temporal gauge, thus setting up the scence for the source-to-solution map. The latter is discussed in Section \ref{sec_sts} where the important
Proposition \ref{prop_source_to_sol} is proved. Section \ref{sec_tcd} computes the equations for the triple cross-derivative when three sources are introduced. Section \ref{sec_muloc} supplies the necessary tools from microlocal analysis needed to compute the symbol of the triple interaction and the latter is computed in Section \ref{sec_proof}.
Section \ref{sec_trivial_centre} proves a result about the structure of Lie algebras with trivial centre, and completes the proof of Theorem \ref{th_main} in the case that $G$ has finite centre. The final Section \ref{sec_general_G} contains the proof of Theorem \ref{th_main} in the general case.

There are three appendices, first of which derives explicit formulas in coordinates, for example, for $d_A^* F_A$. The second appendix discusses the direct problem for the Yang--Mills equations, and the last one gives an elementary alternative to the result in Section \ref{sec_trivial_centre} in the case that $\g = \su(n)$ with $n \ge 2$.

\bigskip

\noindent {\bf Acknowledgements.}  ML was supported by Academy of Finland grants
320113 and 312119.
LO was supported by EPSRC grants EP/P01593X/1 and EP/R002207/1,	
XC and GPP were supported by EPSRC grant EP/R001898/1, and XC was supported by NSFC grant 11701094. 
LO thanks Matthew Towers for discussions of Lie algebras.

\section{Parallel transport}
\label{sec_pt}

We will explain in Section \ref{sec_general_G} how the case of an arbitrary compact, connected Lie group $G$ can be reduced to the case that $G$ has finite centre, that is, the set
    \begin{align*}
Z(G) = \{z \in G : \text{$zh = hz$ for all $h \in G$}\}
    \end{align*}
is finite. In this case, the proof of Theorem \ref{th_main} will ultimately boil down to inversion of a non-abelian broken light ray transform. This transform is the composition of two parallel transports, and we begin by defining the parallel transport used in the paper.

For the moment we may let $(M,g)$ be any Lorentzian manifold, and $G$ any compact matrix Lie group with Lie algebra $\mathfrak{g}$.
However, we will work with trivial bundles for simplicity.
Let $A\in \Omega^{1}(M;\mathfrak{g})$ be a connection and let us first define the parallel transport on the principal bundle $M \times G$
with respect to $A$:
the parallel transport $\U_\gamma^A$ along a
curve $\gamma:[0,T]\to M$ is given by $\U_\gamma^A = U(T)$ where $U$ is the solution of the ordinary differential equation
    \begin{align}\label{master_pt}
\begin{cases}
\dot{U}+\pair{A,\dot{\gamma}(t)}U=0, & t \in [0,T],
\\
U(0)= \id.
\end{cases}
    \end{align}
Here $\pair{\cdot,\cdot}$ is the pairing between covectors and vectors.

In general, if $\mathbb V$ is a vector space and $\rho: G\to \GL(\mathbb V)$ is a linear representation, the parallel transport on the associated vector bundle $M\times \mathbb V$ is defined by
$\P_\gamma^{A,\rho} = \rho(\U_\gamma^A)$.
Two representations will be of importance to us.
First, when $G \subset \GL(\C^n)$ and $\mathbb V = \C^n$ we have the representation given by $\rho=\id$.
In other words, $\P_\gamma^{A,\id} v = \U_\gamma^A v$ for $v \in \mathbb V$. We call this the principal representation.

Second, when $\mathbb V=\mathfrak{g}$
we have the adjoint representation $\rho = \Ad$ where
$\Ad(h)$, $h \in G$, is typically written $\Ad_h$ and defined by
$\Ad_h b = h b h^{-1}$ for $b \in \mathfrak{g}$.
We have
\[
\P_{\gamma}^{A,\Ad} b = \Ad_{\U_\gamma^A} b = \U_\gamma^A b (\U_\gamma^A)^{-1}, \quad b \in \mathfrak{g}.
\]
It is straightforward to verify that $W(t) = U(t) b U^{-1}(t)$
solves
    \begin{align}\label{adjoint_pt}
\begin{cases}
\dot{W}+ [\pair{A,\dot{\gamma}(t)}, W] =0, & t \in [0,T],
\\
W(0)= V,
\end{cases}
    \end{align}
where $U$ is the solution of (\ref{master_pt}).

When $M$ is a convex subset of Minkowski space $\mathbb{R}^{1 + 3}$ and $x, y \in M$, there is a unique geodesic  $\gamma$ from $x$ to $y$, up to reparametrization. The parallel transport $\U_\gamma^A$ does not depend on the parametrization of $\gamma$, and we write simply $\P_{y \gets x}^{A, \rho} = \P_\gamma^{A, \rho}$ in this case.

We are now ready to define the non-abelian broken light ray transforms
used in the proof of Theorem \ref{th_main}.
We write
    \begin{align}\label{def_diamonds_etc}
\mathbb L &= \{(x,y) \in {\mathbb D}^2 : \text{there is a lightlike geodesic joining $x$ and $y$}\},
\\\notag
\mathbb S^+(\mho) &= \{(x,y,z) \in {\mathbb D}^3 : (x,y), (y,z) \in \mathbb L,\ x < y < z,\ x,z \in \mho,\ y \notin \mho \},
 \end{align}
where $x<y$ means that there is a future pointing causal curve from $x$ to $y$.
(For $(x,y) \in \mathbb L$, we have $x<y$ if and only if the time coordinate of $y-x$ is strictly positive.)
Define
\begin{align*}
\mathbf{S}^{A,\rho}_{z \gets y \gets x} &= \mathbf{P}^{A,\rho}_{z \gets y} \mathbf{P}^{A,\rho}_{y \gets x}, \quad (x,y,z) \in \mathbb S^+(\mho).
    \end{align*}
We will reduce the transform $\mathbf{S}^{A,\Ad}_{z \gets y \gets x}$ to $\mathbf{S}^{A,\id}_{z \gets y \gets x}$ as follows:

\begin{lemma}\label{lem_ad_id_reduction}
Suppose that a compact, connected matrix Lie group $G$ has finite centre
and let $A, B \in \Omega^{1}(\mathbb D;\g)$.
If
$\mathbf{S}^{A,\Ad}_{z \gets y \gets x} = \mathbf{S}^{B,\Ad}_{z \gets y \gets x}$ for all $(x,y,z) \in \mathbb S^+(\mho)$
then $\mathbf{S}^{A,\id}_{z \gets y \gets x} = \mathbf{S}^{B,\id}_{z \gets y \gets x}$ for all $(x,y,z) \in \mathbb S^+(\mho)$.
\end{lemma}
\begin{proof}
Let $(x,y,z) \in \mathbb S^+(\mho)$ and $b \in \g$.
Then $\u b = b \u$ where
    \begin{align*}
\u =(\U^B_{z \gets y} \U^B_{y \gets x})^{-1} \U^A_{z \gets y} \U^A_{y \gets x}= \U^B_{x \gets y} \U^B_{y \gets z} \U^A_{z \gets y} \U^A_{y \gets x}.
    \end{align*}
As this holds for all $b \in \g$
we see that $\u$ is in the centre $Z(G)$.
For the convenience of the reader we recall the proof of this well-known fact.
Let $h \in G$. As $G$ is connected, there is a path $H : [0,1] \to G$
satisfying $H(0) = \id$ and $H(1) = h$.
Define the path $F(t) = \u H(t) \u^{-1} H^{-1}(t)$ in $G$.
Then $F(0) = \id$ and
    \begin{align*}
\dot F = \u \dot H \u^{-1} H^{-1} - \u H \u^{-1} H^{-1} \dot H H^{-1}
= \u H H^{-1} \dot H \u^{-1} H^{-1} - \u H \u^{-1} H^{-1} \dot H H^{-1} = 0,
    \end{align*}
where we used the fact that $b = H^{-1} \dot H \in \g$ commutes with $\u^{-1}$.
We conclude that  $\u h \u^{-1} h^{-1} = F(1) = \id$.

Now $\u \in Z(G)$ depends continuously on $x$, $y$ and $z$, and $\u \to \id$ when $y \to x$ and $z \to x$.
As $Z(G)$ is finite, we have $\u = \id$, and therefore
    \begin{align*}
\U^A_{z \gets y} \U^A_{y \gets x} = \U^B_{z \gets y} \U^B_{y \gets x}.
    \end{align*}
\end{proof}

We have previously inverted the transform $\mathbf{S}^{A,\id}_{z \gets y \gets x}$ in the case of the unitary group $G = \mathrm{U}(n)$, see Proposition 2 of \cite{CLOP},
where slightly different choice of $\mho$ and $\mathbb D$ is used. However, the proof works for any matrix Lie group, and also for the present choice of $\mho$ and $\mathbb D$. Moreover, the gauge $\u$ defined in Lemma 3 of \cite{CLOP} is smooth up to $\p \mathbb D$ whenever the two connections $A$ and $B$ are smooth up to $\p \mathbb D$.

Until treating the case of an arbitrary compact, connected Lie group in Section~\ref{sec_general_G},
we will focus on proving:
\begin{proposition}\label{prop_analytic_step}
Suppose that $G$ has finite centre.
If $A$ and $B$ are as in Theorem \ref{th_main}
and if $\mathcal D_A = \mathcal D_B$, then there are $\tilde A \sim A$ and $\tilde B \sim B$ in $\mathbb D$ such that
$\mathbf{S}^{\tilde A,\Ad}_{z \gets y \gets x} = \mathbf{S}^{\tilde B,\Ad}_{z \gets y \gets x}$ for all $(x,y,z) \in \mathbb S^+(\mho)$.
\end{proposition}

Under the additional assumption that $G$ has finite centre, Theorem \ref{th_main} follows then from Proposition \ref{prop_analytic_step}, Lemma \ref{lem_ad_id_reduction}
and the proof of Proposition 2 in \cite{CLOP}.

\section{Yang--Mills equations with a source}
\label{sec_ym}

In this section we let $(M,g)$ be any oriented Lorentzian manifold, and consider the Yang--Mills equations with a source
        \begin{equation}
\label{eq_YM_source}
d^\ast_V F_V = J
\end{equation}
on $M$.
Here the source $J$ cannot be arbitrarily chosen but must obey the compatibility condition
    \begin{align}\label{comp_cond}
d^*_V J = 0
    \end{align}
due to the following well-known lemma. We give a proof for the convenience of the reader.
\begin{lemma}\label{lem_div_J}
Let $V \in C^3(M;T^*M\otimes\mathfrak{g})$.
Then $d^*_V d^*_V F_V = 0$, and the Yang--Mills equations with a source (\ref{eq_YM_source}) imply the compatibility condition (\ref{comp_cond}).
\end{lemma}

\begin{proof}
Since $d_{V}^*= \pm\star d_{V}\star$ we see that given
any $\omega\in \Omega^{k}(M;\mathfrak{g})$ we have
\[(d_{V}^*)^{2}\omega= \pm \star d_{V}\star\star d_{V}\star \omega=\pm \star d_{V}^2\star \omega=\pm\star[F_{V},\star\omega].\]
So it is enough to prove that $[F_{V},\star F_{V}]=0$. But this is a purely algebraic fact that holds for
any $\omega\in \Omega^{2}(M;\mathfrak{g})$, that is,
    \begin{align*}
[\omega,\star\omega]=0, \quad \omega\in \Omega^{2}(M;\mathfrak{g}).
    \end{align*}
This is equivalent with
    \begin{align}\label{eq_bracket_star2}
\omega\wedge\star\omega-\star\omega\wedge\omega=0.
    \end{align}
To check this, write $\omega=\omega_{ij}dx^{i}\wedge dx^{j}$ and note that
\[dx^{i}\wedge dx^{j}\wedge \star(dx^{k}\wedge dx^{l})\neq 0\]
if and only if $i=k$, $j=l$, $i \ne j$ and $k \ne l$. Thus
\[\omega\wedge\star\omega=(\omega_{ij})^{2}dx^{i}\wedge dx^{j}\wedge \star(dx^{i}\wedge dx^{j})\]
and since
\[dx^{i}\wedge dx^{j}\wedge \star(dx^{i}\wedge dx^{j})=\star(dx^{i}\wedge dx^{j})\wedge dx^{i}\wedge dx^{j}\]
$\star\omega\wedge\omega$ has the same expression and (\ref{eq_bracket_star2}) holds.
\end{proof}

The next lemma, proven again for convenience, implies that the source in (\ref{eq_YM_source}) changes to $\U^{-1} J \U$ when a gauge transformation $\U \in C^\infty(M,G)$ acts on $V$.
We use the shorthand notation $B = \U \cdot A$ for (\ref{gauge_equiv}).

\begin{lemma} $B = \U \cdot A$ implies
    \begin{align}\label{gauge_source}
d_{B}^*F_{B}=\U^{-1}d_{A}^*F_{A}\U.
    \end{align}

\end{lemma}
\begin{proof}By assumption
\[B=\U^{-1}d\U+\U^{-1}A\U.\]
A direct calculation from the definitions shows that
\begin{equation}
d_{B}\omega=\U^{-1}d_{A}(\U\omega \U^{-1})\U, \quad \omega \in \Omega^p.
\label{eq:basic}
\end{equation}
Using $d^*_{A}=\star d_{A}\star$ and \eqref{eq:basic} we see that
\[d_{B}^*F_{B}=\U^{-1}d_{A}^*F_{A}\U\]
since $F_{B}=\U^{-1}F_{A}\U$.
\end{proof}

\section{Gauge fixing}
\label{sec_gauges}

Gauge fixing is a mathematical procedure for coping with redundant degrees
of freedom in field variables. Our work uses two gauges, namely the temporal gauge and the relative Lorenz gauge. While these are typical gauge choices, we will give below a self-contained presentation of certain, perhaps less commonly used, properties of these gauges.

\subsection{Temporal gauge}
In this section we write $(x^0, x^1, x^2, x^3) = (t,x) \in \R^{1+3}$ for the Cartesian coordinates.
The signature convention $(-+++)$ is chosen for the Minkowski metric.
A connection $A \in \Omega^1(M;\mathfrak{g})$, with $M \subset \R^{1+3}$, is said to be in the temporal gauge if $A_0 = 0$ where $A = A_\alpha dx^\alpha$.

For a connection $V \in \Omega^1(\mathbb D; \mathfrak g)$ we define a connection $\mathscr T(V)$ in temporal gauge by
    \begin{align}\label{temporal_U}
\mathscr T(V) = \U \cdot V,
\quad \text{where} \quad
\begin{cases}
\p_t \U = -V_0 \U,
\\
\U|_{t = \psi(x)} = \id,
\end{cases}
    \end{align}
and $\psi(x) = |x|-1$.
Observe that $\{(t,x) \in \mathbb D : t = \psi(x)\} = \p^- \mathbb D$ and $\U \in G^0(\mathbb D,p)$. Therefore $\mathscr T(V) \sim V$ in $\mathbb D$.

We shall prove the following uniqueness result:
\begin{proposition}\label{prop_tempg_uniq}
Let $A, B \in C^3(\mathbb D;T^*\mathbb D\otimes\mathfrak{g})$ solve the Yang--Mills equations (\ref{eq:1}) in the set $\mathbb D \setminus \mho$.
Suppose that $d_A^* F_A = d_B^* F_B$ in $\mho$
and that there is $\U \in C^\infty(\mathbb D; G)$ such that $A = \U \cdot B$ near $\p^- \mathbb D$ and that $\U = \id$ in $\mho$ near $\p^- \mathbb D$.
 Suppose, furthermore, that both $A$ and $B$ are in the temporal gauge.
Then $\U$ does not depend on $t$, and $A = \U \cdot B$ in $\mathbb D$.
\end{proposition}

\subsubsection{Reduced equations}

We follow a reduction given in \cite{C-B}.
Suppose that a connection $A \in \Omega^1(M;\mathfrak{g})$ is in temporal gauge and write $d_A^* F_A = J$.
For the convenience of the reader, we give a proof of the following formula, see Lemma \ref{lem_YM_coord} in Appendix \ref{appendix_computations},
    \begin{align*}
d_A^* F_A
&= \left(
\partial_\beta (\partial^\alpha A_\alpha)
-\partial^\alpha \partial_\alpha A_\beta
- [\partial^\alpha  A_\alpha, A_\beta]\right.
\\&\qquad \left.
- 2 [A^\alpha, \partial_\alpha A_\beta]
+ [A^\alpha, \partial_\beta A_\alpha]
- [A^\alpha, [A_\alpha, A_\beta]]\right) dx^\beta.
    \end{align*}
Here, and throughout the paper, indices are raised and lowered by using the Minkowski metric.
Taking $\beta = 0$ we get the constraint equation
    \begin{align}\label{constraint}
\p_0 (\partial^a A_a) + [A^a, \partial_0 A_a] = J_0,
    \end{align}
with $a=1,2,3$, and taking $\beta = j = 1,2,3$ we get
    \begin{align}\label{YM_red}
\partial_j(\partial^a A_a)
-\partial^\alpha \partial_\alpha A_j
+ \tilde N_j(A, \p_x A) = J_j.
    \end{align}
Here $\p_x A = (\p_1 A, \p_2 A, \p_3 A)$ and $\tilde N_j$ contains the terms that are of order one and zero,
    \begin{align*}
\tilde N_j(A, \p_x A) =
-[\partial^a  A_a, A_j] - 2 [A^a, \partial_a A_j] + [A^a, \partial_j A_a] - [A^a, [A_a, A_j]].
    \end{align*}
In the remainder of this section, we will use systematically Greek letters for indices over $0,1,2,3$ and Latin letters for $1,2,3$.

We differentiate (\ref{constraint}) using $\p_j$ and (\ref{YM_red}) using $\p_0$, to obtain
    \begin{align*}
&\p_j \p_0 (\partial^a A_a) = -[\p_j A^a, \partial_0 A_a] - [A^a, \p_j \partial_0 A_a] + \p_j J_0
\\
&
\partial_j \p_0 (\partial^a A_a)
-\partial^\alpha \partial_\alpha \p_0 A_j
+ \p_0 \tilde N_j(A, \p_x A) = \p_0 J_j.
    \end{align*}
Substituting the first equation to the second one gives
    \begin{align}\label{YM_red2}
\Box \p_t A_j + N_j(A, \p_x A, \p_t A, \p_x \p_t A) = \p_t J_j - \p_j J_0,
    \end{align}
where we have written
    \begin{align}\label{def_Minkowski_Box}
\Box = -\p^\alpha \p_\alpha
= \p_t^2 - \p_{x_1}^2 - \p_{x_2}^2 - \p_{x_3}^2,
    \end{align}
and
$$
N_j(A, \p_x A, \p_t A, \p_x \p_t A) =  -[\p_j A^a, \p_0 A_a] - [A^a, \p_j \p_0 A_a] + \p_0 \tilde N_j(A, \p_x A).
$$
We call (\ref{YM_red2}) the reduced Yang--Mills equations.

\subsubsection{Pseudolinearization}
\label{sec_pseudolin}

Observe that for bilinear and trilinear forms $b$ and $m$,
    \begin{align*}
b(A,A) - b(\tilde A, \tilde A)
&=
b(A-\tilde A, A) + b(\tilde A, A-\tilde A),
\\\notag
m(A,A,A) - m(\tilde A, \tilde A, \tilde A)
&=
m(A-\tilde A, A, A) + m(\tilde A, A - \tilde A, A) + m(\tilde A, \tilde A, A - \tilde A).
    \end{align*}
Hence if $A$ and $\tilde A$ satisfy (\ref{YM_red2}) with the same $J$, then the difference $A-\tilde A$ satisfies a linear equation of the form
    \begin{align}\label{psilin}
\Box \p_t (A-\tilde A) + X_1 \p_t (A - \tilde A) + X_2 (A - \tilde A) = 0
    \end{align}
where $X_j$, $j=1,2$, are first order differential operators in the $x^1, x^2$ and $x^3$ variables, with coefficients that depend on $A$ and $\tilde A$, and whence also on the $x^0$ variable.
Writing $u = A-\tilde A$, $Y_1 = -1$ and $Y_2 = 0$, the system (\ref{psilin}) is equivalent to (\ref{model}), with $f_1 = 0$ and $f_2 = 0$, studied in
Appendix \ref{appendix_energy}.

\subsubsection{Proof of Proposition \ref{prop_tempg_uniq}}

$A_0 = 0 = B_0$ implies that
$\U^{-1} \p_t \U = 0$, that is, $\p_t \U = 0$.
Due to its time-independence, $\U$ is well-defined and smooth in whole $\mathbb D$ and $\U = \id$ in $\mho$.
We define $\tilde A = \U \cdot B$ and proceed to show that $A = \tilde A$ in $\mathbb D$.

As $\tilde A$ is gauge equivalent to $B$, the Yang--Mills equations $d_{\tilde A} F_{\tilde A} = 0$ hold in $\mathbb D \setminus \mho$.
As $\U = \id$ in $\mho$, we have $\tilde A = B$ in $\mho$.
Therefore $d_{\tilde A} F_{\tilde A} = d_{A} F_{A}$ in $\mho$.
As $\U$ does not depend on $t$, we see that $\tilde A_0 = 0$.
Hence $A$ and $\tilde A$ are two solutions
to the reduced Yang--Mills equations (\ref{YM_red2}), with the same $J$,
and the difference $A - \tilde A$ satisfies (\ref{psilin}).
As they also coincide near $\p^- \mathbb D$, Lemma \ref{lem_fsop} in Appendix \ref{appendix_energy} implies that $A=\tilde A$ in $\mathbb D$.

\subsection{Relative Lorenz gauge}

For a moment we may let $(M,g)$ be any oriented Lorentzian manifold of even dimension.
Consider two connections $A$ and $V$ on $M$ solving the Yang--Mills equations without (\ref{eq:1})
and with (\ref{eq_YM_source}) a source, respectively. That is,
$d_A^* F_A = 0$ and $d_V^* F_V = J$.
We will rewrite the latter equation in terms of the difference $W = V - A$.

Directly from the definition of curvature
\[F_{V}=d(W+A)+\frac{1}{2}[W+A,W+A]=F_{A}+dW+[A,W]+[W,W]/2\]
and thus
\begin{equation}
\label{eqn : perturbed curvature} F_V = F_A + d_A W + [W, W]/2.
\end{equation}
Since $d_A^\ast = \star d_A \star$ it follows that $d_V^\ast = d_A^\ast + \star [W, \star \cdot]$. Combining this with \eqref{eqn : perturbed curvature} and $d_A^* F_A = 0$, we see that
$d_V^* F_V = J$ is equivalent with
 \begin{equation}\label{eq_YM_relative}
d_A^\ast d_A W + \star[W, \star F_A] + \mathcal N(W) = J,
\end{equation}
where the non-linear part reads
    \begin{align}\label{def_nonlin}
\mathcal N(W) = \frac{1}{2} d_A^\ast [W, W] + \star[W, \star d_A W]  + \frac{1}{2} \star[W, \star[W, W]].
    \end{align}

We say that $V \in \Omega^1(M;\mathfrak{g})$ is in the Lorenz gauge relative to a background connection $A \in \Omega^1(M;\mathfrak{g})$ if $d_A^\ast V=d_A^\ast A$.
In this case (\ref{eq_YM_relative}) is
equivalent with
\begin{equation}\label{eq_YM_relative_L}
\Box_A W + \star[W, \star F_A] + \mathcal N(W) = J,
\end{equation}where $\Box_A = d_A d^\ast_A + d^\ast_A d_A$ is the connection wave operator.

The semilinear wave equation (\ref{eq_YM_relative_L}), together with suitable initial conditions, is solvable when the source $J$ is small and smooth enough, see, for example, (the proof of) Theorem 6 in \cite{Kato1975}.
However, its solution $W$ solves the actual Yang--Mills equations (\ref{eq_YM_relative}) if and only if $d_A d_A^\ast W=0$.
Recall also that if $W$ solves (\ref{eq_YM_relative}), or equivalently (\ref{eq_YM_source}), then $J$ satisfies the compatibility condition (\ref{comp_cond}).
We will therefore study the system combining
(\ref{comp_cond}) and (\ref{eq_YM_relative_L}).
Observe that (\ref{comp_cond}) is equivalent with
    \begin{align}\label{eq_J0}
\p_t J_0 + [A_0, J_0] + [W_0, J_0] = \p^j J_j + [A^j, J_j] + [W^j, J_j],
    \end{align}
where $j=1,2,3$. This can be viewed as an ordinary differential equation for $J_0$.

We begin with an uniqueness result that is similar to Proposition \ref{prop_tempg_uniq}.
For $r > 0$ and $x \in \R^{1+3}$ we define the
rescaled and translated diamond
    \begin{align*}
\mathbb D(x,r) = \{ry + x : y \in \mathbb D \}.
    \end{align*}

\begin{lemma}\label{lem_fsop_Lorenz}
Let $r > 0$ and $x \in \R^{1+3}$ and write $\tilde{\mathbb D} = \mathbb D(x,r)$.
Let $A \in \Omega^1(\tilde{\mathbb D},\g)$ and suppose that $W_{(\ell)}, J_{(\ell)} \in C^2(\tilde{\mathbb D};T^*\tilde{\mathbb D}\otimes\mathfrak{g})$
solve
    \begin{align*}
\begin{cases}
\Box_A W + \star[W, \star F_A] + \mathcal N(W) = J,
\\
d_A^* J + \star [W, \star J] = 0,
\end{cases}
    \end{align*}
in $\tilde{\mathbb D}$ for $\ell=1,2$. Suppose, furthermore,
that $W_{(\ell)}, J_{(\ell)}$, $\ell=1,2$, vanish near $\p^- \tilde{\mathbb D}$
and that the spatial parts of $J_{(1)}$ and $J_{(2)}$ of coincide on $\tilde{\mathbb D}$, that is,
$J_{(1),j} = J_{(2),j}$ for $j=1,2,3$.
Then $W_{(1)} = W_{(2)}$ and $J_{(1)} = J_{(2)}$ in $\tilde{\mathbb D}$.
\end{lemma}
\begin{proof}
Pseudolinearization analogous to that in Section \ref{sec_pseudolin}
shows that the difference $(W_{(1)} - W_{(2)}, J_{(1)} - J_{(2)})$
solves a system of the form (\ref{model}) in Appendix \ref{appendix_energy} with $f_1 = 0$ and $f_2 = 0$. The coefficients of this system depend on $W_{(\ell)}, J_{(\ell)}$ and they satisfy the assumptions of Lemma \ref{lem_fsop} in Appendix \ref{appendix_energy}.
Lemma \ref{lem_fsop} is formulated for $\mathbb D$ rather than for $\tilde{\mathbb D}$, however, the form of the system (\ref{model}) is invariant under a rescaling and translation. Therefore Lemma \ref{lem_fsop} holds also for $\tilde{\mathbb D}$ and we conclude by applying it.
\end{proof}

We will now turn to existence of solutions to the Yang--Mills equations. It is convenient work in the cylinder $M = (-2,2) \times \R^3$ containing the diamond $\mathbb D$, rather than in $\mathbb D$.
Let us consider again the system combining (\ref{comp_cond}) and (\ref{eq_YM_relative_L}),
    \begin{align}\label{YM_Lorenz_with_cc}
\begin{cases}
\Box_A W + \star[W, \star F_A] + \mathcal N(W) = J, & t \geq -1,
\\
d_A^* J + \star [W, \star J] = 0, & t \geq -1,
\\
W = 0,\ J = 0, & t \le -1.
\end{cases}
    \end{align}

\begin{lemma}\label{lem_Lorenz}
Let $A \in \Omega^1(M; \g)$ and suppose that $W,J \in C^3(M;T^*M\otimes\mathfrak{g})$ solve (\ref{YM_Lorenz_with_cc}). Suppose moreover that $A$ solves (\ref{eq:1}) in $\mathbb D$ and that $\supp(J_j)$, $j=1,2,3$, is contained in the interior of $\mathbb D$. Then
$W$ solves (\ref{eq_YM_relative}) in $\mathbb D$, with $J$ on the right-hand side.
\end{lemma}
\begin{proof}
The equations (\ref{eq_YM_relative})
and (\ref{eq_YM_relative_L}) differ by the term $d_A d_A^* W$ on the left-hand side.
Hence it is enough to verify that $H = 0$ in $\mathbb D$ where $H = d^*_A W$.
We write $V = W + A$.
As $A$ solves (\ref{eq:1}) in $\mathbb D$,
$d_V^* F_V$ coincides with the left-hand side of (\ref{eq_YM_relative}) in $\mathbb D$, and
the first equation in (\ref{YM_Lorenz_with_cc}), in other words (\ref{eq_YM_relative_L}), implies that
$d^*_V F_V + d_A H = J$ in $\mathbb D$.
Applying $d^*_V$ to this equation, we have using Lemma \ref{lem_div_J} and the second equation in (\ref{YM_Lorenz_with_cc}) that
$d_V^* d_A H = 0$ in $\mathbb D$. This is a linear wave equation for $H$.
We will show below that $W$ vanishes near $\p^- \mathbb D$. Hence also $H$ vanishes near $\p^- \mathbb D$, and as it satisfies the linear wave equation, it vanishes in the whole $\mathbb D$. This type of finite speed of propagation result is of course standard, and it follows also from Lemma \ref{lem_fsop}  Appendix \ref{appendix_energy}.

Let us now show that $W$ vanishes near $\p^- \mathbb D$.
There is $r \in (0,1)$ such that
$\supp(J_j) \subset \mathbb D(0,r)$
for $j=1,2,3$.
Let $\tilde{\mathbb D}$ in Lemma~\ref{lem_fsop_Lorenz} satisfy $\tilde{\mathbb D} \cap \mathbb D(0,r) = \emptyset$
and $\p^- \tilde{\mathbb D} \subset \{t < -1\}$. Lemma~\ref{lem_fsop_Lorenz} implies that $W = 0$ in $\tilde{\mathbb D}$ by comparison with the trivial solution.
By varying $\tilde{\mathbb D}$ we see that
$W$ vanishes in $\{t \le 0\} \setminus \mathbb D(0,r)$, and also near $\p\mathbb D \cap \{t=0\}$. In particular, $W$ vanishes near $\p^- \mathbb D$.
\end{proof}

\begin{remark}\label{rem_supp_J0}
As the second equation in (\ref{YM_Lorenz_with_cc})
is equivalent with
the ordinary differential equation (\ref{eq_J0}), we see that if
$\supp(J_j) \subset (0,T) \times K$, $j=1,2,3$, for some $K \subset \R^3$, then also $\supp(J_0) \subset (0,T) \times K$
for a solution of (\ref{YM_Lorenz_with_cc}).
\end{remark}

We prove the following result in Appendix \ref{appendix_energy}.

\begin{proposition}
\label{prop_direct}
Suppose that $A \in \Omega^1(M; \g)$ is bounded, together with all its derivatives, and let $k \ge 4$.
Then there is a neighbourhood $\mathcal H$ of the zero function in $H^{k+2}(M;\mathfrak g)$
such that for all $J_j \in \mathcal H$, $j=1,2,3$, there
is a unique solution
    \begin{align*}
W \in H^{k+1}(M; T^*M \otimes \mathfrak g), \quad
J_0 \in H^{k+1}(M; \mathfrak g)
    \end{align*}
of (\ref{YM_Lorenz_with_cc}) with $J=J_0 dx^0 + \dots + J_3 dx^3$.
Moreover, the map $(J_1, J_2, J_3) \mapsto (W,J_0)$
is smooth from $\mathcal H^3$ to $H^{k+1}(M; T^*M \otimes \mathfrak g \oplus \mathfrak g)$.
\end{proposition}

\section{Source-to-solution map}
\label{sec_sts}

We begin with a lemma, that will be used only once, and that
highlights the difference between the pointed gauge group $G^0(\mathbb D, p)$ and the full gauge group $G(\mathbb D)$.

\begin{lemma}\label{lem_mod_data}
Suppose that $\tilde A \sim A$ near $\p^- \mathbb D$ and consider the modified data set
    \begin{align*}
\tilde{\mathcal D}_A = \{ V' \in \mathcal D_A :\ &\text{$V' = \tilde{A}$ in $\mho$ near $\p^- \mathbb D$} \}.
    \end{align*}
Let $V' \in \tilde{\mathcal D}_A$.
Then there are $\U \in G^0(\mathbb D, p)$ and
$V  \in C^3(\mathbb D; T^* \mathbb D \otimes \mathfrak{g})$
such that $V' = V|_\mho$, $V = \U \cdot \tilde{A}$ near $\p^- \mathbb D$,
and $\U = \id$ in $\mho$ near $\p^- \mathbb D$.
\end{lemma}
\begin{proof}
It follows immediately from the definitions of the sets $\mathcal D_A$ and $\tilde{\mathcal D}_A$ that there are $\U \in G^0(\mathbb D, p)$ and
$V  \in C^3(\mathbb D; T^* \mathbb D \otimes \mathfrak{g})$
such that  $V' = V|_\mho$, $V = \U \cdot \tilde{A}$ near $\p^- \mathbb D$, and $V = \tilde A$ in $\mho$ near $\p^- \mathbb D$.
Then $\U$ satisfies
    \begin{align}\label{stabilization}
\U \cdot \tilde A = \tilde A
    \end{align}
in $\mho$ near $\p^- \mathbb D$.
As (\ref{stabilization}) is equivalent with the differential equation
$d\U = [\tilde{A}, \U]$,
and $\U(p) = \id$, it follows that $\U = \id$ in $\mho$ near $\p^- \mathbb D$.
\end{proof}

If we used gauge equivalence with respect to $G(\mathbb D)$ in the definition $\mathcal D_A$, then (\ref{stabilization}) would still hold in a neighbourhood $\mathcal U \subset \overline\mho$ of $\p^- \mathbb D \cap \overline\mho$, however, this simply says that $\U|_{\mathcal U}$ is in the stabilizer subgroup
$\{\U \in C^\infty(\mathcal U; G) : \U \cdot \tilde A = \tilde A\}$ with respect to $\tilde A|_{\mathcal U}$.
In general, the stabilizer subgroup may be non-trivial.

Recall that the temporal gauge version $\mathscr T(V)$ of a connection $V$ is defined by (\ref{temporal_U}).
Recall, furthermore, that the system (\ref{YM_Lorenz_with_cc}) of Yang--Mills equations in relative Lorenz gauge with the compatibility condition is posed on $M = (-2,2) \times \R^3$.

\begin{proposition}\label{prop_source_to_sol}
Suppose that $A \in \Omega^{1}(\mathbb D;\mathfrak{g})$ satisfies (\ref{eq:1}) in $\mathbb D$.
Then there is a connection $\tilde{A} \in \Omega^{1}(\mathbb D; \mathfrak{g})$
such that $\tilde{A} \sim A$ in $\mathbb D$, $\tilde{A}|_\mho$ is in temporal gauge,
and the following holds:
for all $x \in \mho$ there are a neighbourhood $\mho_0 \subset \mho$
of $x$ and a neighbourhood $\mathcal H$ of
the zero function in $H_0^7(\mho_0;\mathfrak g)$
such that $\mathcal D_A$ determines $\tilde{A}|_\mho$ and the source-to-solution map
    \begin{align*}
L(J_1,J_2,J_3) = \mathscr T(V)|_\mho, \quad J_j \in \mathcal H,\ j=1,2,3,
    \end{align*}
where $V = W + \tilde{A}$ and $(W,J_0)$ is the solution of (\ref{YM_Lorenz_with_cc}) with $J=J_0 dx^0 + \dots + J_3 dx^3$ and with $A$
replaced by an arbitrary smooth, compactly supported extension of $\tilde{A}$ to $M$.
\end{proposition}
\begin{proof}
Let $\tilde{A}' \in \mathcal D_A$ be in the temporal gauge and satisfy $d_{\tilde{A}'}^*F_{\tilde{A}'}=0$ in $\mho$.
Such $\tilde{A}'$ exists, for example, $\tilde{A}' = \mathscr T(A)|_\mho$ is a possible choice.
There is $\tilde{A}$ such that $\tilde{A}' = \tilde{A}|_\mho$, $d_{\tilde{A}}^*F_{\tilde{A}}=0$ in $\mathbb D$ and $\tilde{A} \sim A$ near $\p^- \mathbb D$.
Proposition \ref{prop_abstract_uniq_pointed} in Appendix \ref{appendix_energy} implies that $\tilde{A} \sim A$ in $\mathbb D$.
Choose a smooth, compactly supported extension of $\tilde A$ in $M$, still denoted by $\tilde A$.

For $x \in \mho$ we choose $\epsilon > 0$ small enough so that $\mathbb D(x,\epsilon) \subset \mho$ and let $\mho_0$ be the interior of $\mathbb D(x,\epsilon)$.
Let $t_0$ be the time coordinate of $x$.
Let $J_j \in H^7_0(\mho_0; \mathfrak g)$, $j=1,2,3$, be small, and
consider the solution $(W,J_0)$ of the system (\ref{YM_Lorenz_with_cc})
with $A = \tilde{A}$ in $(-1,t_0) \times \R^3$.
This solution vanishes outside $\mho_0$ and near $\p^- \mathbb D(x,\epsilon)$,
and it does not depend on $\tilde{A}$ away from $\mho_0$.
The vanishing of $(W,J_0)$ outside $\mho_0$  and near $\p^- \mathbb D(x,\epsilon)$ is shown similarly to the vanishing of $W$ near $\p^- \mathbb D$ in the proof of Lemma~\ref{lem_Lorenz}, and we omit this argument.
To see that $(W,J_0)$ does not depend on $\tilde{A}$ away from $\mho_0$, we consider two solutions to (\ref{YM_Lorenz_with_cc})
with different backgrounds $A$ in $(-1,t_0 + \epsilon) \times \R^3$.
Both the backgrounds are assumed to coincide with $\tilde A$ in $\mho_0$.
As both the solutions vanish near $\p^- \mathbb D(x,\epsilon)$, Lemma \ref{lem_fsop_Lorenz} implies that they are identical in $\mathbb D(x,\epsilon)$.

Extending $(W,J_0)$ by zero we get a solution in the set $\mho_- = \mho \cap \{t < t_0\}$.
To summarize, the solution $(W,J_0)$ in $\mho_-$ is determined by $\tilde{A}'$ and our choice of $J_j$, $j=1,2,3$.
Defining a connection $\hat V = \hat V(J_1, J_2, J_3)$ on $\mho_-$
by $\hat V = W + \tilde{A}$ we have $d_{\hat V}^* F_{\hat V} = J$ in $\mho_-$ where $J = J_0 dx^0 + \dots +J_3 dx^3$.
We write $\mho_+ = \mho \cap \{t > t_0\}$, and consider the set
    \begin{align*}
\mathcal L = \mathcal L(J_1,J_2,J_3) = \{ \mathscr T(V') &: \text{$V' \in \tilde{\mathcal D}_A$, $V' = \hat V$ in $\mho_-$,}
\\&\text{and the spatial part of $d_{V'}^* F_{V'}$ vanishes in $\mho_+$} \}.
    \end{align*}
Here $\mathscr T$ is defined by (\ref{temporal_U}) with $|x| < \epsilon_0$, cf. (\ref{def_mho}). No confusion should arise from our use of $\mathscr T$ for temporal gauge both in $\mho$ and in $\mathbb D$ since $\mathscr T(V|_\mho) = \mathscr T(V)|_\mho$ for a connection $V$ on $\mathbb D$.

As $\hat V$ is determined by $\mathcal D_A$ (and the choice of $\tilde A'$),
also $\mathcal L$ is determined by $\mathcal D_A$.
Moreover, $\mathscr T(V)|_\mho \in \mathcal L$
where $V = W + \tilde{A}$ and $(W,J_0)$ is the solution of (\ref{YM_Lorenz_with_cc}) in $M$ with $J_j$, $j=1,2,3$, as above and $A=\tilde{A}$.
The solution $(W,J_0)$ in $M$ is an extension of the solution $(W,J_0)$ in $(0,t_0) \times \R^3$, which justifies our reuse of symbols.
Observe that Proposition \ref{prop_direct}, together with the Sobolev embedding theorem, guarantees that $W \in C^3(\mathbb D; T^* \mathbb D \otimes \mathfrak{g})$, and that Remark \ref{rem_supp_J0} guarantees that $\supp(J_0) \subset \mho$.

To conclude the proof, it remains to show that $\mathcal L$ consists of a single element.
Suppose that $W', \tilde W' \in \mathcal L$.
By Lemma \ref{lem_mod_data} there are connections $V$, $\tilde V$
and gauges $\u$, $\tilde \u$
satisfying
$W' = \mathscr T(V)|_\mho$, $\tilde W' = \mathscr T(\tilde V)|_\mho$,
$d_{V}^*F_{V}=0 = d_{\tilde V}^*F_{\tilde V}$ in $\mathbb D \setminus \mho$, $V = \u \cdot\tilde{A}$ and $\tilde V = \tilde \u \cdot \tilde A$ near $\p^- \mathbb D$, and $\u = \id = \tilde \u$ in $\mho$ near $\p^- \mathbb D$.
We define
    \begin{align*}
\begin{cases}
\p_t \U = -V_0 \U,
\\
\U|_{t = \psi(|x|)} = \id,
\end{cases}
\quad
\begin{cases}
\p_t \tilde \U = -\tilde V_0 \tilde \U,
\\
\tilde \U|_{t = \psi(|x|)} = \id,
\end{cases}
    \end{align*}
and set $W = \U\cdot V$ and $\tilde W = \tilde \U \cdot \tilde V$.
Then $W_0 = 0 = \tilde W_0$ in $\mathbb D$.
Moreover, it follows from the definition of  $\mathscr T$
that $W' = W|_{\mho}$ and $\tilde W' = \tilde W|_{\mho}$.

There holds $V = \tilde A = \tilde V$ in $\mho$ near $\p^- \mathbb D$. This implies $\U = \tilde \U$ and $W = \tilde W$ in $\mho$ near $\p^- \mathbb D$.
Writing $\U_- = \U\u\tilde \u^{-1}\tilde \U^{-1}$, we have that $W = \U_- \cdot \tilde W$ near $\p^- \mathbb D$ and $\U_-= \id$ in $\mho$ near $\p^- \mathbb D$.

In fact, as $V = \hat V = \tilde V$ in $\mho_-$, we have $\U = \tilde \U$ and $W = \tilde W$ in $\mho_-$.
Hence also $d_{W}^* F_W = d_{\tilde W}^* F_{\tilde W}$ in $\mho_-$.
The spatial parts of $d_{V}^* F_V$ and $d_{\tilde V}^* F_{\tilde V}$
vanish in $\mho_+$. As gauge transformations act componentwise on $d_{W}^* F_W$, see (\ref{gauge_source}),
also the spatial parts of $d_{W}^* F_W$ and $d_{\tilde W}^* F_{\tilde W}$
vanish in $\mho_+$.
Writing $J_0$ for the temporal part of
$d_{W}^* F_W$,
the compatibility condition $d_W^* d_{W}^* F_W = 0$, see Lemma \ref{lem_div_J}, together with $W_0 = 0$, implies that $\p_t J_0 = 0$ in $\mho_+$.
The same holds for $\tilde J_0$, the temporal part of $d_{\tilde W}^* F_{\tilde W}$. But $J_0 =\tilde J_0$ on $\mho \cap \{t=t_0\}$, and hence $J_0 =\tilde J_0$ in $\mho_+$. To summarize $d_{W}^* F_W = d_{\tilde W}^* F_{\tilde W}$ in $\mho$.
Proposition \ref{prop_tempg_uniq} implies that $W = \tilde W$ in $\mho$. In other words $W' = \tilde W'$ and this is the only element in $\mathcal L$.
\end{proof}

\section{Linearization of the Yang--Mills equations in Lorenz gauge}
\label{sec_tcd}

Let us study multiple-fold linearizations of  (\ref{eq_YM_relative_L}).
Consider a three-parameter family
    \begin{align*}
(W,J) = (W(\epsilon), J(\epsilon)), \quad \epsilon = (\epsilon_{(1)}, \epsilon_{(2)}, \epsilon_{(3)}),
    \end{align*}
of solutions to (\ref{eq_YM_relative_L}), vanishing for $t \le 0$, where $\epsilon$ is in a neighbourhood of the origin in $\R^3$.
Assume that the source term is linear in the sense that $J = \sum_{k = 1}^3 \epsilon_{(k)} J_{(k)}$ for some $J_{(k)} \in \Omega^1(\R^{1+3}; \mathfrak g)$.
Writing
    \begin{align}\label{def_Y}
Y_{(k)} = \frac{\partial W}{\partial \epsilon_{(k)}}\bigg|_{\epsilon = 0},
\quad
Y_{(kl)} = \frac{\partial^2 W}{\partial \epsilon_{(k)}\partial \epsilon_{(l)}}\bigg|_{\epsilon = 0},
\quad
Y_{(123)} = \frac{\partial^3 W}{\partial \epsilon_{(1)}\partial \epsilon_{(2)}\partial \epsilon_{(3)}}\bigg|_{\epsilon = 0},
    \end{align}
and differentiating
\eqref{eq_YM_relative_L} in $\epsilon$ gives the following system of linear wave equations
\begin{equation}\label{eqn : linearized Y-M}
\begin{cases}
\Box_A Y_{(k)}     + \star [Y_{(k)}, \star F_A]  = J_{(k)}, & t \geq 0,
\\
\Box_A Y_{(kl)} + \star [Y_{(kl)}, \star F_A] + N(2)   = 0, & t \geq 0,
\\
\Box_A Y_{(123)} + \star [Y_{(123)}, \star F_A] + N(3) = 0, & t \geq 0,
\\
Y_{(k)} = Y_{(kl)} = Y_{(123)} = 0, & t \leq 0,
\end{cases}
\end{equation}
where the nonlinear terms read
    \begin{align*}
N(2) &= \frac{1}{2}d_A^\ast [Y_{(k)}, Y_{(l)}] + \frac{1}{2}d_A^\ast [Y_{(l)}, Y_{(k)}] + \star [Y_{(k)}, \star d_A Y_{(l)} ] + \star [Y_{(l)}, \star d_A Y_{(k)} ],
    \end{align*}
and, writing $S_3$ for the set of permutations on $\{1,2,3\}$,
    \begin{align*}
N(3) &= \frac{1}{2} \sum_{\pi\in S_3} \bigg( \frac{1}{2}d_A^\ast [Y_{(\pi(1)\pi(2))}, Y_{(\pi(3))}] + \frac{1}{2}d_A^\ast [Y_{(\pi(1))}, Y_{(\pi(2)\pi(3))}]
\\& \qquad\quad
+  \star [Y_{(\pi(1)\pi(2))}, \star d_A Y_{(\pi(3))} ]    +  \star [Y_{(\pi(1))}, \star d_A Y_{(\pi(2)\pi(3))} ]
\\ & \qquad\quad
+ 2 \star [Y_{(\pi(1))}, \star [Y_{(\pi(2))}, Y_{(\pi(3))}]]\bigg).
    \end{align*}

Now we continue the calculation in Cartesian coordinates in Minkowski space $\R^{1+3}$, and use the formulas
    \begin{align}\label{eqn : dastarliebracket}
d_A^\ast [X, Z] &= [d_A^* X, Z]-[X, d_A^* Z]
\\\notag&\qquad+
[\p^\alpha X_\beta + [A^\alpha, X_\beta], Z_\alpha]  dx^\beta - [X_\alpha, \p^\alpha Z_\beta + [A^\alpha, Z_\beta]] dx^\beta,
\\\label{eqn : wstardaw}
\star [X, \star d_A Z]
&=
-[X^\alpha, \partial_\alpha Z_\beta + [A_\alpha, Z_\beta]] dx^\beta
+[X^\alpha, \partial_\beta Z_\alpha + [A_\beta, Z_\alpha]] dx^\beta,
\\\label{eqn : cubic term}
\star [X, \star [Y, Z]]
&=- [X^\alpha,   [Y_\alpha, Z_\beta]] dx^\beta + [X^\alpha,   [Y_\beta, Z_\alpha]] dx^\beta.
    \end{align}
These formulas are derived in Appendix \ref{appendix_computations}.
Using \eqref{eqn : dastarliebracket}--\eqref{eqn : cubic term}
and the Lorenz gauge condition $d_A^* W = 0$, we rewrite the first three equations in \eqref{eqn : linearized Y-M}, modulo lower order terms, as follows
    \begin{align}\label{eqn : 1-fold interaction}
\Box_A Y_{(k)} &= J_{(k)},\\
\label{eqn : 2-fold interaction} \Box_A Y_{(kl)} &= \tilde{N}(2), \\
\label{eqn : 3-fold interaction} \Box_A Y_{(123)} &= \tilde{N}(3),
    \end{align}
where the components of the right-hand sides of the last two equations read
    \begin{align*}
\tilde{N}_\beta(2)
&=  2 [Y_{(k)}^\alpha, \partial_\alpha Y_{(l), \beta}]
- [Y_{(k)}^\alpha, \partial_\beta Y_{(l), \alpha}]
 + 2 [Y_{(l)}^\alpha,   \partial_\alpha Y_{(k), \beta}]
 - [Y_{(l)}^\alpha,   \partial_\beta Y_{(k), \alpha}],
\\
\tilde{N}_\beta(3)
&=
\frac{1}{2}\sum_{\pi \in S_3} \bigg(
2[Y_{(\pi(1)\pi(2))}^\alpha,   \partial_\alpha(Y_{(\pi(3)), \beta})]   - [Y_{(\pi(1)\pi(2))}^\alpha,   \partial_\beta(Y_{(\pi(3)), \alpha})   ]
\\&\notag \qquad\quad+
2[Y_{(\pi(1))}^\alpha,   \partial_\alpha(Y_{(\pi(2)\pi(3)), \beta})]   - [Y_{(\pi(1))}^\alpha,   \partial_\beta(Y_{(\pi(2)\pi(3)), \alpha})   ]
\\&\notag \qquad\quad
+ 4 [Y_{(\pi(1))}^\alpha, [Y_{(\pi(2)), \alpha}, Y_{(\pi(3)), \beta}] ]    \bigg).
    \end{align*}

\section{Preliminaries on microlocal analysis}
\label{sec_muloc}

\subsection{Distributions associated to conormal bundles and two Lagrangians}

The advantage of working in the relative Lorenz gauge is that the Yang--Mills equations reduces to a cubic nonlinear wave equation with the linear part given by the connection wave operator $\Box_A$, modulo zeroth order terms.
The parametrix for $\Box_A$ is a distribution associated to an intersecting pair of Lagrangians (shortly an IPL distribution), in the sense of \cite{Melrose-Uhlmann-CPAM1979},
and we use the product calculus of conormal distributions to study the non-linear part.

The proof of Proposition \ref{prop_analytic_step} in the next section relies solely on symbolic computations, and we
recall here only that conormal and IPL distributions have principal symbols and that the corresponding symbol maps are isomorphisms, modulo lower order terms in a suitable sense. We will not recall the definitions of these classes of distributions, them being somewhat technical, instead we refer the reader to \cite{CLOP} for a review of the theory that we use and that was originally developed in \cite{Hormander-FIO1, Duistermaat-Hormander-FIO2, Melrose-Uhlmann-CPAM1979}.
Even the precise definition of spaces of symbols is not important for our present purposes, since we will consider only symbols that are positively homogeneous in the fibre variable.

Recall that a pseudodifferential operator $A$ on a manifold $X$ with a homogeneous principal symbol $a$ is said to be elliptic at $(x,\xi) \in T^\ast X \setminus 0$ if $a(x,\xi) \ne 0$.
The wavefront set $\WF(u) \subset T^\ast X \setminus 0$ of a distribution $u$ on $X$ is the complement of its regular set, whilst the regular set consists of such points $(x,\xi) \in T^\ast X \setminus 0$ that there is a zeroth order pseudodifferential operator $A$ that is elliptic at $(x,\xi)$ and that satisfies $Au \in C^\infty(X)$.
We denote by $\singsupp(u)$ the projection of $\WF(u)$ on $X$, and by $\WF(A)$ the essential support of $A$, that is, the projection of $\WF(\mathscr A) \subset (T^*X \setminus 0)^2$ on the first factor $T^*X\setminus 0$ where $\mathscr A$ is the Schwartz kernel of $A$.
Moreover, we say that $A$ is a microlocal cutoff near $(x,\xi) \in T^\ast X \setminus 0$ if $A$ is elliptic at $(x,\xi)$ and $\WF(A)$ is contained in a small neighbourhood of $\{(x, \lambda \xi) : \lambda > 0\}$.

Let $E$ be a complex smooth vector bundle over $X$ and $\Omega^{1/2}$ the half density bundle. A conormal distribution $u \in I^m(N^\ast Y; E \otimes \Omega^{1/2})$ of order $m \in \R$ is a compactly supported distribution taking values on the tensor bundle $E \otimes \Omega^{1/2}$ with $\text{WF}(u)$ contained in the conormal bundle $N^\ast Y$ of a submanifold $Y$ of $X$.
In addition, $u$ is required to have certain local structure on $Y$, see (2.4.1) in \cite{Hormander-FIO1}, precise form of which is not important for our purposes.
What is important is that the principal symbol $\sigma[u]$ of $u$ is a smooth section of $E \otimes \Omega^{1/2}$, invariantly defined on $N^\ast Y \setminus 0$, and that the principal symbol map $u \mapsto \sigma[u]$ gives the short exact sequence,
\begin{align}\label{short_exact}
0 \rightarrow I^{m-1}(N^\ast Y; E \otimes \Omega^{1/2}) \hookrightarrow  I^m(N^\ast Y; E \otimes \Omega^{1/2}) \\\notag \xrightarrow{\;\sigma\;} S^{m + n/4}/S^{m + n/4 - 1}(N^\ast Y; E \otimes \Omega^{1/2}) \rightarrow 0,
\end{align}
see \cite[Theorem 2.4.2]{Hormander-FIO1} and \cite[Theorem 18.2.11]{Hormander-Vol3}.
Here $n$ is the dimension of $X$ and $S^{m}(N^\ast Y; E \otimes \Omega^{1/2})$, with $m \in \R$, is the space of symbols, see \cite[Definition 18.2.10]{Hormander-Vol3}. For our purposes it suffices to note that positively homogeneous sections of degree $m$ are in this space, and that if $\Omega^{1/2}$ is trivialized by choosing
a nowhere vanishing positively homogeneous section $\mu$ of degree $r$,
then $\sigma[u]$ is positively homogeneous of degree $m + r$ if
$$(\mu^{-1} \sigma[u])(x, \lambda \xi) = \lambda^m (\mu^{-1} \sigma[u]) (x, \xi), \qquad \mbox{for any $\lambda > 0$ and $(x,\xi) \in N^* Y \setminus 0$}.$$
Since the half density is involved here, the given homogeneity looks a little different from the classical definition in \cite[p.67]{Hormander-Vol3}.

More generally, a Lagrangian distribution $u \in I^{m}(\Lambda; E \otimes \Omega^{1/2})$ is a compactly supported distribution  with $\text{WF}(u)$ contained in a conical Lagrangian submanifold $\Lambda$ of $T^\ast X \setminus 0$, and certain local structure, see (3.2.14) in \cite{Hormander-FIO1}. Its principal symbol is invariantly defined on $\Lambda$
as a smooth section of the bundle $E \otimes \Omega^{1/2} \otimes L$, where $L$ is the Maslov bundle over $\Lambda$.
Analogously to (\ref{short_exact}) the principal symbol map gives an isomorphism
    \begin{align*}
I^m(\Lambda; E \otimes \Omega^{1/2}) \to S^{m + n/4}(\Lambda; E \otimes \Omega^{1/2} \otimes L)
    \end{align*}
modulo lower order terms, see \cite[Theorem 3.2.5]{Hormander-FIO1}. We write also
    \begin{align*}
I(\Lambda; E) = \bigcup_{m \in \R} I^{m}(\Lambda; E \otimes \Omega^{1/2}).
    \end{align*}

The notion of Lagrangian distributions is insufficient to completely describe the fundamental solution of wave equations as two Lagrangian manifolds are needed in order to describe the propagating singularities and the singularities at the source. An IPL distribution $u \in I^{m}(\Lambda_0, \Lambda_1; E \otimes \Omega^{1/2})$ is
compactly supported distribution with $\text{WF}(u)$ contained in $\Lambda_0\cup\Lambda_1$, where $(\Lambda_0, \Lambda_1)$ is a cleanly intersecting pair of conical Lagrangian submanifolds of $T^\ast X \setminus 0$, and with certain local structure on $\Lambda_0 \cup \Lambda_1$, see \cite{Melrose-Uhlmann-CPAM1979}. Here $\Lambda_1$ is a manifold with boundary, while $\Lambda_0$ is a manifold without boundary, and by cleanly intersecting, we mean $$\Lambda_0 \cap \Lambda_1 = \partial \Lambda_1, \quad
T_\lambda (\Lambda_0) \cap T_\lambda (\Lambda_1) = T_\lambda(\partial \Lambda_1).$$
Again what we really need in the present paper is the symbol map for such distributions. In this case the symbol map is an isomorphism, modulo lower order terms, from
$I^{m}(\Lambda_0, \Lambda_1; E \otimes \Omega^{1/2})$
to the space
$$  \left\{ (a^{(1)}, a^{(0)}) \,\bigg|
\begin{array}{l} a^{(0)} \in S^{m-1/2+n/4}(\Lambda_0 \setminus \partial \Lambda_1; E \otimes \Omega^{1/2} \otimes L),\\
a^{(1)} \in S^{m+n/4}(\Lambda_1; E\otimes\Omega^{1/2} \otimes L),\\
a^{(1)}|_{\partial \Lambda_1} = \mathscr{R} a^{(0)},\\
\mbox{$ha^{(0)}$ is smooth up to $\partial \Lambda_1$ if $h$ vanishes on $\partial \Lambda_1$.}
\end{array}  \right\}.$$
We remark that $\mathscr{R}$ maps the $E\otimes\Omega^{1/2} \otimes L$-valued symbols over $\Lambda_0$ to the $E\otimes\Omega^{1/2} \otimes L$-valued symbols over $\Lambda_1$ and acts as a multiplication by a scalar on $E$.

If $(x,\xi) \in \Lambda_j \setminus \p \Lambda_1$ for $j=0$ or $j=1$, then there is a microlocal cutoff $\chi$ near $(x,\xi)$ such that $\chi u \in I(\Lambda_j; E)$
for all $u \in I^{m}(\Lambda_0, \Lambda_1; E \otimes \Omega^{1/2})$.
The only place where we need the full picture of IPL distributions, instead of the above microlocal reduction to Lagrangian distributions, is equation (\ref{eqn : symbol transport equation at the intersection}) giving an initial condition on $\p \Lambda_1$ for a transport equation on $\Lambda_1$.
Moreover,
apart from (\ref{eqn : symbol transport equation at the intersection}), we can also avoid the use of Lagrangian distributions in favour of conormal distributions, since all the Lagrangian manifolds $\Lambda_0$ and $\Lambda_1$ considered below will be conormal bundles away from $\p \Lambda_1$.

The principal symbol $\sigma[\Box_A]$ and the subprincipal symbol $\sigma_{\text{sub}}[\Box_A]$ read
$$\sigma[\Box_A](x, \xi) =  \xi^\alpha \xi_\alpha,  \quad \sigma_{\text{sub}}[\Box_A](x, \xi) =  2 \imath^{-1} [\xi^\alpha A_\alpha, \cdot].$$
We denote by $\Phi_s$, $s \in \R$, the flow of  the Hamilton vector field $H_{\sigma[\Box_A]}$ of $\sigma[\Box_A]$, and define for a subset $\mathscr B$ of the characteristic set $\Sigma$ of $\Box_A$ the future flowout of $\mathscr B$ by
\begin{align}
\label{def_flowout}
\{ (y,\eta) \in \Sigma;\
(y,\eta) = \Phi_s(x,\xi),\ s \in \R,\ (x,\xi) \in \mathscr B,\ y \ge x \}.
\end{align}

As $\Box_A$ is of real principal type one can use
the theory by
H\"{o}rmander and Duistermaat \cite{Duistermaat-Hormander-FIO2}
to understand its parametrix.
A completely symbolic parametrix construction, based on IPL distributions,
was given by Melrose and Uhlmann \cite{Melrose-Uhlmann-CPAM1979},
and
the following adaptation of their construction in the vector valued case can be found in \cite{CLOP}:
\begin{proposition}
Let $\Lambda_0$ be a conormal bundle such that $H_{\sigma[\Box_A]}$ is nowhere tangent to $\Lambda_0$.
Denote by $\Lambda_1$ the future flowout of $\Lambda_0 \cap \Sigma$.
Consider the wave equation \begin{equation}\label{eqn : Cauchy for wave}\left\{ \begin{array}{ll}
\Box_A u = f,& \mbox{in $\mathbb{R}^{1 + 3}$}\\ u|_{t<0} = 0,&
\end{array} \right. \end{equation}where $f \in I(\Lambda_0; E)$ and $E = T^\ast \mathbb{R}^{1+3} \otimes \mathfrak{g}$.
Then $u \in \bigcup_{m \in \R} I^{m}(\Lambda_0, \Lambda_1; E \otimes \Omega^{1/2})$ and the corresponding principal symbols satisfy
\begin{eqnarray}\label{eqn : symbol transport equation}(\mathscr{L}_{H_{\sigma[\Box_A]}} + \imath \sigma_{\text{sub}}[\Box_A]) \sigma[u] = 0 && \mbox{on $\Lambda_1 \setminus \Lambda_0$},\\ \label{eqn : symbol transport equation at the intersection}\sigma[u] = \mathscr{R}((\sigma[\Box_A])^{-1} \sigma[f]) && \mbox{on $\Lambda_1 \cap \Lambda_0$}.\end{eqnarray}
Here $\mathscr{L}_{H_{\sigma[\Box_A]}}$ denotes the Lie derivative with respect to $H_{\sigma[\Box_A]}$.
\end{proposition}

We will compute symbols related to the non-linear terms by using the following result, implicitly contained in \cite{Greenleaf-Uhlmann-CMP-1993} and explicitly formulated for example in \cite{CLOP}.
\begin{proposition}
Let $K_{(1)}$ and $K_{(2)}$ be two transversal submanifolds of $X$, let
    \begin{align*}
(x,\xi) \in N^*(K_{(1)} \cap K_{(2)}) \setminus (N^*K_{(1)} \cup N^*K_{(2)}),
    \end{align*}
and let
$u_{(j)} \in I(N^* K_{(j)}; E)$, $j=1,2$.
If $\chi$ is a microlocal cutoff near $(x,\xi)$ and $\mu$ is a nowhere vanishing half density on $X$, then writing $u_{(1)} u_{(2)} = \mu (\mu^{-1}u_{(1)}) (\mu^{-1} u_{(2)})$,
there holds $\chi(u_{(1)}u_{(2)}) \in I(N^*(K_{(1)} \cap K_{(2)}); E)$ and
\begin{equation}
\label{eqn : symbol calculus of conormal} \sigma[\chi(u_{(1)}u_{(2)})](x, \xi) = \mu^{-1}(x)\sigma[\chi](x, \xi) \sigma[u_{(1)}](x, \xi_1) \sigma[u_{(2)}](x, \xi_{(2)}),
\end{equation} where $\xi = \xi_{(1)} + \xi_{(2)}$ with $\xi_{(1)} \in N^\ast K_{(1)}$ and $\xi_{(2)} \in N^\ast K_{(2)}$.
\end{proposition}

\subsection{Parallel transport for the principal symbol}

As in \cite{CLOP}, the transport equation \eqref{eqn : symbol transport equation} can be understood as a parallel transport equation as in Section \ref{sec_pt},
\begin{align*}
\p_s \hat u_\alpha + [\pair{A, \dot \gamma}, \hat u_\alpha] = 0, \quad \hat u_\alpha(s) = e^{\varrho(s)}(\mu^{-1}\sigma[u_\alpha])(\bm{\beta}(s)),\quad u = u_\alpha dx^\alpha.
\end{align*}
Here $\mu$ is a nowhere vanishing half density on $\Lambda_1 \setminus \Lambda_0$,
$\bm{\beta}(s) = (\gamma(s), \dot \gamma^*(s))$, with $\dot\gamma^* = \dot\gamma_\alpha dx^\alpha$,
is the bicharacteristic curve
 emanating from $\bm{\beta}(0) \in \Lambda_0 \cap \Lambda_1$,
 and
    \begin{align}\label{def_rho}
 \varrho(s) = \int_{0}^s (\mu^{-1}\mathscr{L}_{H_{\sigma[\Box_A]}} \mu)(\bm{\beta}(r)) dr.
    \end{align}
Comparing with (\ref{adjoint_pt}), we see that the 1-form components $\hat u_\alpha$ satisfy the parallel transport equation on $M \times \mathfrak g$ corresponding to the adjoint representation of $G$.
In particular, if $x,y \in \mathbb L$ and the singular support of $f$ does not intersect the line segment from $x$ to $y$, then
    \begin{align}\label{pt_symbol}
e^{\varrho(s)}(\mu^{-1}\sigma[u_\alpha])[u_\alpha](y, \xi) =  \P_{y \gets x}^{A,\Ad} \left((\mu^{-1}\sigma[u_\alpha])[u_\alpha](x, \xi)\right),
    \end{align}
where $\xi$ is the covector corresponding to the direction of the line segment, and $\bm{\beta}$ in (\ref{def_rho}) satisfies $\bm{\beta}(0) = (x,\xi)$ and $\bm{\beta}(s) = (y,\xi)$.

We will also need the fact that positive homogeneity is preserved in (\ref{pt_symbol}) in the sense of the following proposition, where we have fixed a nowhere vanishing half density $\mu$ of degree $1/2$ on $\Lambda_1 \setminus \Lambda_0$.

\begin{proposition}
Let $u \in I(\Lambda_0, \Lambda_1; T^\ast \mathbb{R}^{1+3} \otimes \mathfrak{g} \otimes \Omega^{1/2})$ be an IPL distribution solving \eqref{eqn : Cauchy for wave} and its symbol $\sigma[u]$ positively homogeneous of degree $q + 1/2$ on $\Lambda_1 \setminus \Lambda_0$.
Suppose that $\Lambda_1 \setminus \Lambda_0 = N^\ast K \setminus 0$ for some $K \subset \mathbb{R}^{1+3}$. Then for any $(y, \xi) \in N^\ast K \setminus 0$ with $(y, \xi) = \Phi_s(x, \xi)$ for some $s \in \mathbb{R}$, we have \begin{equation}\label{eqn : parallel transport for homogeneous}
e^{\varrho(s)}(\mu^{-1}\sigma[u])(y, \pm \lambda \xi) = \lambda^q  \P_{y \gets x}^{A,\Ad}  ((\mu^{-1}\sigma[u])(x, \pm \xi) ), \quad \mbox{for any $\lambda > 0$}.
\end{equation}
\end{proposition}
Recall that $\Phi_s$ is the flow of the Hamilton vector field $H_{\sigma[\Box_A]}$. For the proof, the reader is referred to  our work \cite[Proposition 1]{CLOP}.

\section{Proof of Proposition \ref{prop_analytic_step}}
\label{sec_proof}

We follow the construction in \cite{CLOP}, however, the analysis in the present paper is more involved due to the non-linearity in Yang--Mills equations being more complicated than the simple cubic non-linearity considered in \cite{CLOP}, and also due to the gauge invariance of the Yang--Mills equations.
We will focus on the new features of the proof and refer to \cite{CLOP} for technical details that are unchanged.

In order to apply the microlocal machinery in Section \ref{sec_muloc} we need to consider the Yang--Mills equations on the tensor product bundle $T^* \R^{1+3} \otimes \g \otimes \Omega^{1/2}$.
This is achieved by choosing a nowhere vanishing half density $\mu$ on $\R^{1+3}$ and by considering
the conjugated operator $\mu^{-1} P(\mu W)$ instead of
$P(W) = \Box_A W + \star[W, \star F_A] + \mathcal N(W)$, cf. (\ref{eq_YM_relative_L}).
In fact, we choose $\mu$ so that $\mu=1$ identically in the Cartesian coordinates,
and to simplify the notation, we omit writing $\mu$ in what follows. However, we warn the reader that additional determinant factors appear in other coordinates. These can be included in the factors $\tilde \alpha_{(k)}$ in (\ref{Jk_symb}), and
$\alpha_{(k)}$, $\alpha_{(kl)}$ and $\alpha$
in (\ref{Y_rescaled}).

Recall that $\mathbb S^+(\mho)$ is defined by (\ref{def_diamonds_etc}).
Let $(x_{(1)},y,z) \in \mathbb S^+(\mho)$ and consider the line segments
$\gamma_{y \gets x_{(1)}}$ and $\gamma_{z \gets y}$ from $x_{(1)}$ to $y$ and from $y$ to $z$, respectively.
We write
    \begin{align*}
\eta = \dot \gamma_{z \gets y}^*(0), \quad
\xi_{(1)} = \dot \gamma_{y \gets x_{(1)}}^*(\ell),
    \end{align*}
where $\ell \in \R$ satisfies $\gamma_{y \gets x_{(1)}}(\ell) = y$
and $\cdot^* : T_y \R^{1+3} \to T_y^* \R^{1+3}$ denotes the tangent-cotangent isomorphism given by the Minkowski metric. After rescaling $\eta$ and $\xi_{(1)}$, and after a rotation in $\R^3$, we may assume that
    \begin{align}\label{def_eta_xi}
\eta = (1, -a(r), r, 0), \quad \xi_{(1)} = (1,1,0,0),
    \end{align}
where $a(r) = \sqrt{1-r^2}$ and $r \in (-1,1)$.
Then we let $s > 0$ be small and set
    \begin{align}\label{xi23}
\xi_{(2)} = (1,a(s),s,0), \quad \xi_{(3)} = (1,a(s),-s,0).
    \end{align}
The rationale behind this choice of $\xi_{(k)}$, $k=2,3$, is that now $\eta$ can be written as the linear combination
    \begin{align*}
\eta = \kappa_{(1)} \xi_{(1)} + \kappa_{(2)} \xi_{(2)} + \kappa_{(3)} \xi_{(3)},
    \end{align*}
where the scalars $\kappa_{(k)}$ are given explicitly by
       \begin{align}\label{def_kappas}
\kappa_{(1)} = 1 - \frac{1 + a(r)}{1 - a(s)},
\quad
\kappa_{(2)} = \frac{1 + a(r)}{2(1 - a(s))} + \frac{1}{2} \frac{r}{s},
\quad
\kappa_{(3)} = \frac{1 + a(r)}{2(1 - a(s))} - \frac{1}{2} \frac{r}{s}.
    \end{align}

Writing $\gamma(\cdot; x, \xi)$ for the geodesic on $\R^{1+3}$ with the initial conditions $\gamma(0; x, \xi) = x$ and $\dot\gamma^*(0;x, \xi) = \xi$, we define
    \begin{align*}
x_{(k)} = \gamma(-\ell;y,\xi_{(k)}), \quad k=2,3.
    \end{align*}
Then $x_{(2)}, x_{(3)} \in \mho$ for small enough $s > 0$.

It turns out that in the coordinates satisfying (\ref{def_eta_xi})--(\ref{xi23}) it is enough to use sources with all but the $dx^2$ component vanishing.
Let $b_{(k)} \in \mathfrak g$ and set
    \begin{align}\label{the_sources}
J_{(k),2} = J_{(k),2}(s) = b_{(k)} \chi_{(k)} \delta_{x_{(k)}},
\quad k=1,2,3,
    \end{align}
where $\delta_{x_{(k)}}$ is the Dirac delta distribution at $x_{(k)}$
and $\chi_{(k)}$
is a microlocal cutoff near $(x_{(k)}, \pm\xi_{(k)})$. Here the the sign is chosen to be that of $\kappa_{(k)}$, that is, $-$ for $k=1$ and $+$ for $k=2,3$. Moreover, $\chi_{(k)}$ is chosen so that
\begin{itemize}
\item[($\chi$1)] the principal symbol $\sigma[\chi_{(k)}]$ is positively homogeneous of degree $q$;
\item[($\chi$2)] $\supp(J_{(k),2}) \subset \mho_{(k)}$
where $\mho_{(k)} \subset \mho$ is a neighbourhood of $x_{(k)}$,
and for all $k \ne l$ it holds that $x_{(l)} \notin \mathcal J^+(\mho_{(k)})$ where
    \begin{align*}
\mathcal J^+(\mho_{(k)}) = \{y \in \R^{1+3} :
\text{$x < y$ or $x = y$ for some $x \in \mho_{(k)}$} \};
    \end{align*}
\item[($\chi$3)]
$\hat \mho_{(k)} \cap \Gamma_{(l)} = \emptyset$ for all $k \ne l$ where
    \begin{align*}
\hat \mho_{(k)} &= \{(t,x') \in \R^{1+3}: (\tilde t, x') \in \mho_{(k)} \text{ for some $\tilde t \in \R$} \},
\\
\Gamma_{(k)} &= \{\gamma(\tilde t; x_{(k)}, \xi) : \tilde t \in \R,\ (x_{(k)}, \xi) \in \WF(\chi_{(k)})\}.
    \end{align*}
\end{itemize}
The degree $q \in \R$ is chosen negative enough so that $J_{(k),2} \in H_0^7(\mho; \mathfrak g)$.
The geometric setting is shown in Figure \ref{fig_geom}.

\begin{figure}
\centering
\includegraphics[width=0.5\textwidth]{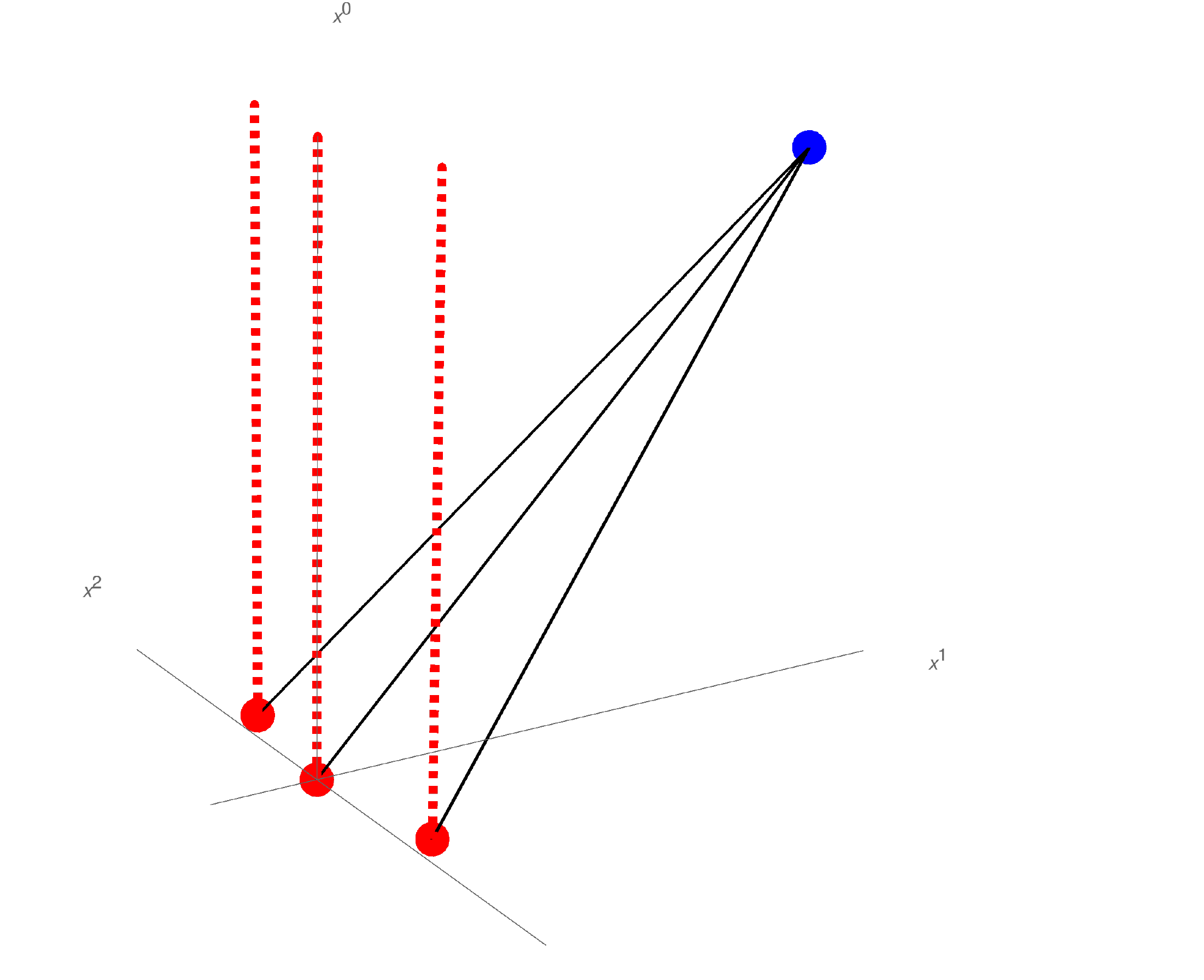}
\caption{Three line segments (in black) along the lightlike geodesics $\gamma_{y \gets x_{(k)}}$ from $x_{(k)}$ (in red) to $y$ (in blue), $k=1,2,3$, in the hyperplane $x^3=0$.
Coordinates are chosen so that (\ref{def_eta_xi})--(\ref{xi23}) hold and that $x_{(1)}$ is at the origin. All three points $x_{(k)}$ are in the plane $x^0=0$, and there exist neighbourhoods $\Omega_{(k)}$ of $x_{(k)}$ so that ($\chi$2) holds.
The set $\hat \Omega_{(k)}$ is a small neighbourhood of the dashed red line through $x_{(k)}$ (in particular, $\hat \Omega_{(1)}$ is a neighbourhood of the $x^0$-axis),
and $\Gamma_{(k)}$ is a small neighbourhood of the black line through $x_{(k)}$, for small $\Omega_{(k)}$ and $\WF(\chi_{(k)})$, hence ($\chi$3) holds.}
\label{fig_geom}
\end{figure}

\begin{proposition}\label{prop_muloc}
Let $x_{(1)},y,z$ and $\eta$, as well as, $b_{(k)}$ and $J_{(k),2}(s)$, with $k=1,2,3$ and small $s>0$, be as above, and define for $\epsilon_{(k)} \in \R$, $k=1,2,3$,
    \begin{align}\label{def_J2}
J_{2}(\epsilon,s) = \epsilon_{(1)} J_{(1),2}(s) + \epsilon_{(2)} J_{(2),2}(s) + \epsilon_{(3)} J_{(3),2}(s), \quad \epsilon = (\epsilon_{(1)}, \epsilon_{(2)}, \epsilon_{(3)}).
    \end{align}
Let $\tilde A$ and $L$ be as in Proposition \ref{prop_source_to_sol}.
Suppose that $r \ne 0$ in (\ref{def_eta_xi}), $b_{(2)} = b_{(3)}$.
Then for any $s_0>0$, the following point values of symbols
    \begin{align*}
\sigma[\p_{\epsilon_{(1)}}\p_{\epsilon_{(2)}}\p_{\epsilon_{(3)}}L(0, J_{2}(0, s), 0)](z,\eta), \quad s \in (0,s_0),
    \end{align*}
determine $\mathbf{S}^{\tilde A,\Ad}_{z \gets y \gets x_{(1)}}[b_{(2)}, [b_{(1)}, b_{(2)}]]$.
\end{proposition}

As $(x_{(1)},y,z) \in \mathbb S^+(\mho)$ and $b_{(1)}, b_{(2)} \in \mathfrak g$ can be chosen arbitrarily apart from the constraint $r \ne 0$, Proposition \ref{prop_analytic_step} follows from Propositions \ref{prop_source_to_sol} and \ref{prop_muloc} together with
Proposition \ref{prop_nestedcommutators} in Section \ref{sec_trivial_centre} below.
Here the case $r=0$ follows by continuity.

For the convenience of readers who do not wish to enter into theory of Lie algebras,
we have included an elementary alternative to Proposition \ref{prop_nestedcommutators} in the case $\g = \su(n)$, with $n \ge 2$, see Lemma \ref{lem_span_sun} in Appendix \ref{sec_sun}.
This special case is interesting in view of the $\SU(3) \times \SU(2) \times \mathrm{U}(1)$ gauge group of the standard model.

We will proceed to give a proof of Proposition \ref{prop_muloc} in Sections \ref{sec_muloc_reduction}--\ref{sec_prinsymb_gauge}.

\subsection{Microlocal reduction from (\ref{YM_Lorenz_with_cc}) to (\ref{eq_YM_relative_L})}
\label{sec_muloc_reduction}

Let $J_{(k),2}$, $k=1,2,3$, be as in (\ref{the_sources}), and write $J_2 = J_2(\epsilon,s)$ for the function defined by (\ref{def_J2}).
To simplify the notation, we write
$J_j = J_{(k),j} = 0$ for $k=1,2,3$ and $j=1,3$,
and, for the remainder of this section, somewhat abusively $A = \tilde A$ where
$\tilde{A}$ is as in Proposition \ref{prop_source_to_sol}.
Then we denote by
    \begin{align}\label{the_sol}
(W,J_0) = (W(\epsilon),J_0(\epsilon)), \quad \epsilon = (\epsilon_{(1)}, \epsilon_{(2)}, \epsilon_{(3)}),
    \end{align}
the solution of (\ref{YM_Lorenz_with_cc}) with $J_j$, $j=1,2,3$, as above and $\epsilon$ near the origin of $\R^3$.
The derivatives of $W$ with respect to $\epsilon$ are denoted by $Y_{(k)}$, $Y_{(kl)}$ and $Y_{(123)}$ as in (\ref{def_Y}),
and we write also
    \begin{align*}
\rho_{(k)} =  \frac{\partial J_0}{\partial \epsilon_{(k)}}  \bigg|_{\epsilon = 0},
\quad
\rho_{(kl)} =  \frac{\partial^2 J_0}{\partial \epsilon_{(k)}\epsilon_{(l)}}  \bigg|_{\epsilon = 0},
\quad
\rho_{(123)} =  \frac{\partial^3 J_0}{\partial \epsilon_{(1)}\epsilon_{(2)}\epsilon_{(3)}}  \bigg|_{\epsilon = 0}.
    \end{align*}
For notational convenience, we translate the origin in (\ref{YM_Lorenz_with_cc}) so that the initial conditions are given at $t=0$ rather than at $t=-1$.

Recall that the second equation in (\ref{YM_Lorenz_with_cc})
is equivalent with (\ref{eq_J0}).
Differentiating \eqref{eq_J0} with respect to $\epsilon_{(k)}$ for $k=1, 2, 3$ gives
    \begin{align}\label{eq_rhok}
\partial_t \rho_{(k)} + [A_0, \rho_{(k)}] =  \partial^j J_{(k),j} + [A^j , J_{(k),j}].
    \end{align}
Writing
    \begin{align*}
\xi = (\tau, \xi') = (\xi_0, \xi_1, \xi_2, \xi_3) \in T_x^* \R^{1+3},
\quad x = (t,x') =  (x^0,x^1,x^2,x^3)\in \R^{1+3},
    \end{align*}
the operator $\p_t$ is elliptic away from its characteristic set $\{\tau = 0\} \subset T^* \R^{1+3}$.
The wave front set of the right-hand side of (\ref{eq_rhok})
is contained in a small neighbouhood of $\{(x_{(k)}, \lambda \xi_{(k)}) : \lambda \ne 0\}$,
and therefore it is disjoint from $\{\tau = 0\}$.
It follows that $\rho_{(k)} \in I(N^*\{x_{(k)}\}; \mathfrak g)$
since the right-hand side of (\ref{eq_rhok}) is in this class.
Recalling the form of $\xi_{(k)}$, $k=1,2,3$, see (\ref{def_eta_xi}) and (\ref{xi23}), symbol evaluation gives
    \begin{align*}
\sigma[\rho_{(1)}](x_{(1)}, -\xi_{(1)}) = 0,
\quad
\sigma[\rho_{(k)}](x_{(k)}, \xi_{(k)}) = (-1)^k s \sigma[J_{(k),2}](x_{(k)}, \xi_{(k)}), \quad k=2,3.
    \end{align*}
Hence $Y_{(k)}$ solves \eqref{eqn : 1-fold interaction}
with $J_{(k)}$ satisfying
    \begin{align}\label{Jk_symb}
J_{(k)} \in I(N^*\{x_{(k)}\}; T^* \R^{1+3} \otimes \mathfrak g),
\quad \sigma[J_{(k)}](x_{(k)}, \pm \xi_{(k)}) = \tilde \alpha_{(k)} b_{(k)} \omega_{(k)},
    \end{align}
where the sign is that of $\kappa_{(k)}$, $\tilde \alpha_{(k)} = \sigma[\chi_{(k)}](x_{(k)}, \pm \xi_{(k)}) \ne 0$, $b_{(k)}$ is as in (\ref{the_sources}), and
    \begin{align*}
\omega_{(1)} = dx^2,
\quad
\omega_{(k)} = (-1)^k s dx^0 + dx^2,
\quad
k=2,3.
    \end{align*}
It follows that away from $x_{(k)}$,
    \begin{align*}
Y_{(k)} \in I(N^* K_{(k)}; T^* \R^{1+3} \otimes \mathfrak g),
    \end{align*}
where $N^\ast K_{(k)}$ is the bicharacteristic flowout emanating from $(x_{(k)}, \xi_{(k)})$. In other words, writing $x_{(k)} = (t_{(k)}, x_{(k)}')$,
$$
K_{(k)} = \left\{(t_{(k)} + s, x_{(k)}' + s \theta) \in \mathbb{R}^{1+3} : |\theta| = 1, s > 0 \right\}.
$$
Moreover, $\singsupp(Y_{(k)}) \subset \Gamma_{(k)}$.

The second derivative of \eqref{eq_J0} in $\epsilon$ for distinct $k, l =1, 2, 3$ reads
    \begin{align}\label{rho_kl}
\partial_t \rho_{(kl)} + [A_0, \rho_{(kl)}] = - [Y_{(k),0}, \rho_{(l)}] - [Y_{(l),0}, \rho_{(k)}] + [Y^j_{(l)}, J_{(k),j}] + [Y^j_{(k)}, J_{(l),j}].
    \end{align}
As $\supp(J_{(k),j}) \subset \mho_{(k)}$ by ($\chi$2),
it follows from (\ref{eq_rhok}) and $J_0 = 0$ for $t \le 0$ that
$\supp(\rho_{(k)}) \subset \hat \mho_{(k)}$.
We see that $Y_{(k)}$ is smooth in the support of $\rho_{(l)}$
for distinct $k$ and $l$, since $\hat \mho_{(k)} \cap \Gamma_{(l)} = \emptyset$ by ($\chi$3).
Moreover, $Y_{(k)}$ solves \eqref{eqn : 1-fold interaction}
with vanishing initial conditions and with the source satisfying $\supp(J_{(k)}) \subset \hat \mho_{(k)} \subset \mathcal J^+(\mho_{(k)})$,
whence $\supp(Y_{(k)}) \subset \mathcal J^+(\mho_{(k)})$ due to finite speed of propagation (as discussed in the proof of Lemma \ref{lem_Lorenz} finite speed of propagation follows from Lemma \ref{lem_fsop} in Appendix \ref{appendix_energy}). As $\singsupp(\rho_{(l)}) = \{x_{(l)}\}$,
it follows from ($\chi$2) that $\rho_{(l)}$ is smooth in the support of $Y_{(k)}$ for distinct $k$ and $l$.
Analogously, $Y_{(k)}$ is smooth in $\supp(J_{(l)})$
and $J_{(l)}$ is smooth in $\supp(Y_{(k)})$ for $k \ne l$.
Therefore the right-hand side of (\ref{rho_kl}) is smooth,
and so is $\rho_{(kl)}$.
This again implies that $Y_{(kl)}$ satisfies \eqref{eqn : 2-fold interaction}
modulo smooth terms.

The third derivative of \eqref{eq_J0} in $\epsilon$ can be written as
    \begin{align*}
\partial_t \rho_{(123)} + [A_0, \rho_{(123)}]
&=
\frac 1 2 \sum_{\pi \in S_3} \bigg(-[ Y_{(\pi(1)\pi(2)), 0}, \rho_{(\pi(3))}]
- [ Y_{(\pi(1)), 0}, \rho_{(\pi(2)\pi(3))}]
\\&\qquad+
[ Y^j_{(\pi(1)\pi(2))}, J_{(\pi(3)),j}] \bigg).
    \end{align*}
It follows from \cite[Th. 8.2.10]{H1} that, for distinct $k$ and $l$, any $(x,\xi) \in \WF(Y_{(k)}Y_{(l)})$ with lightlike $\xi$
satisfies $(x,\xi) \in \WF(Y_{(j)})$ for $j=k$ or $j=l$.
Then \eqref{eqn : 2-fold interaction} implies that
    \begin{align*}
\singsupp(Y_{(kl)})
\subset \singsupp(Y_{(k)}) \cup \singsupp(Y_{(l)}).
    \end{align*}
Similarly with the above, we see also that $\supp(Y_{(kl)}) \subset \mathcal J^+(\mho_{(k)}) \cup \mathcal J^+(\mho_{(l)})$
and $\supp(\rho_{(kl)}) \subset \hat \mho_{(k)} \cup \hat \mho_{(l)}$
for $k \ne l$.
As above, this implies that $\rho_{(123)}$ is smooth,
and that $Y_{(123)}$ satisfies \eqref{eqn : 3-fold interaction}
modulo smooth terms.

\subsection{Principal symbols of interacting waves}\label{subsec : principal symbol}

The linearized equation \eqref{eqn : 2-fold interaction} has source $\tilde{N}(2)$ that consists of products of solutions $Y_{(k)}$, $k=1,2,3$, to the linear wave equation \eqref{eqn : 1-fold interaction}. These products can be viewed as the interactions of waves $Y_{(k)}$ and $Y_{(l)}$. Then the solution $Y_{(kl)}$ to \eqref{eqn : 2-fold interaction} describes the linear waves emanating from the source of such interacting waves $Y_{(k)}$ and $Y_{(l)}$. Analogously the solution $Y_{(123)}$ to \eqref{eqn : 3-fold interaction} describes waves emanating from interaction of $Y_{(1)}$, $Y_{(2)}$ and $Y_{(3)}$.

As $\xi_{(k)}$, $k=1,2,3$, are linearly independent, the submanifolds $K_{(k)}$, $k=1,2,3$, intersect transversally at $y$,
and we may compute the principal symbols $\sigma[Y_{(123)}](y,\eta)$ using the product formula \eqref{eqn : symbol calculus of conormal}.
This requires using the direct sum decomposition
    \begin{align*}
\eta = \eta_{(1)} + \eta_{(2)} + \eta_{(3)} \in
N^\ast_y  K_{(1)} \oplus N^\ast_y K_{(2)} \oplus N^\ast_y K_{(3)},
    \end{align*}
where $\eta_{(k)} = \kappa_{(k)} \xi_{(k)}$ and the scalars $\kappa_{(k)}$ are given by (\ref{def_kappas}).
We will omit below the details related to the choices of the microlocal cutoff when applying \eqref{eqn : symbol calculus of conormal}. The same choices as in \cite{CLOP} can be used, see (54) there and its proof.

By (\ref{eqn : parallel transport for homogeneous}) the incoming principal symbols satisfy
    \begin{align*}
\sigma[Y_{(k)}](y, \eta_{(k)}) = \alpha_{(k)} |\kappa_{(k)}|^{q-1} \P_{y \gets x_{(k)}}^{A,\Ad} b_{(k)} \omega_{(k)},
    \end{align*}
where the scalar factors $\alpha_{(k)}$ converge in $\C \setminus 0$ as $s \to 0$. The factors $\alpha_{(k)}$ are independent from $A$, and their precise form is not important for our purposes.
We refer to \cite{CLOP} for more detail on how to compute these factors. Let us point out, however, that typically $\alpha_{(k)} \ne \tilde \alpha_{(k)}$, with $\tilde \alpha_{(k)}$ as in (\ref{Jk_symb}), due to a contribution from $\mathscr R$ and $\sigma[\Box_A]^{-1}$ in (\ref{eqn : symbol transport equation at the intersection}).

We use the shorthand notations
    \begin{align}\label{Y_rescaled}
\hat Y_{(j)} &=
(\alpha_{(j)})^{-1} |\kappa_{(j)}|^{1-q} \sigma[Y_{(j)}](y, \eta_{(j)}),
\\\notag
\hat Y_{(kl)} &=
-\imath (\alpha_{(kl)})^{-1} |\kappa_{(k)}\kappa_{(l)}|^{1-q} \sigma[Y_{(kl)}](y, \eta_{(kl)}),
\\\notag
\hat Y_{(123)} &=
-\alpha^{-1} |\kappa_{(1)}\kappa_{(2)}\kappa_{(2)}|^{1-q} \sigma[Y_{(123)}](y, \eta),
    \end{align}
where $\eta_{(kl)} = \eta_{(k)} + \eta_{(l)}$,
$\alpha_{(kl)} = \alpha_{(k)}\alpha_{(l)}$, and $\alpha = \iota \alpha_{(1)}\alpha_{(2)}\alpha_{(3)}$. The constant $\iota \in \mathbb C \setminus 0$ comes from (\ref{eqn : symbol transport equation at the intersection}) and is independent from $A$.
Then
\begin{multline*}
\hat{Y}_{(k l),\beta} =
p^{-1}(y, \eta_{(k l)}) \left( 2 \eta_{(l), \alpha} [\hat{Y}_{(k)}^\alpha,    \hat{Y}_{(l), \beta}  ]  - \eta_{(l), \beta} [\hat{Y}_{(k)}^\alpha,   (\hat{Y}_{(l), \alpha})   ]
\right.\\\left. + 2 \eta_{(k), \alpha} [\hat{Y}_{(l)}^\alpha,    \hat{Y}_{(k), \beta}   ]   - \eta_{(k), \beta} [\hat{Y}_{(l)}^\alpha,   \hat{Y}_{(k), \alpha}   ] \right),
\end{multline*}
where $p(y, \xi) = - \xi_0^2 + \xi_1^2 + \xi_2^2 + \xi_3^2$.
Writing
    \begin{align}\label{two_fold_expansion}
\hat Y_{(kl),\beta} = c_{(kl),\beta} p^{-1}(y, \eta_{(k l)}) [\tilde b_{(k)}, \tilde b_{(l)}],
\quad
\tilde b_{(j)} = \P_{y \gets x_{(j)}}^{A,\Ad} b_{(j)},
    \end{align}
we have
    \begin{align*}
c_{(12),0}&=\kappa_{(1)}+2 \kappa_{(2)} s^2-\kappa_{(2)},\quad
c_{(12),1}=\kappa_{(1)}-a(s) \kappa_{(2)},\quad
c_{(12),2}=2 \kappa_{(1)} s+\kappa_{(2)} s,\\
c_{(13),0}&=\kappa_{(1)}+2 \kappa_{(3)} s^2-\kappa_{(3)},\quad
c_{(13),1}=\kappa_{(1)}-a(s) \kappa_{(3)},\quad
c_{(13),2}=-2 \kappa_{(1)} s-\kappa_{(3)} s,
    \end{align*}
and
    \begin{align*}
c_{(23),0}&=-3 \kappa_{(2)} s^2+\kappa_{(2)}+3 \kappa_{(3)}
s^2-\kappa_{(3)},\\
c_{(23),1}&=a (s) \kappa_{(2)} s^2+a (s) \kappa_{(2)}-a (s)
\kappa_{(3)} s^2-a (s) \kappa_{(3)},\\
c_{(23),2}&=\kappa_{(2)} s^3-3 \kappa_{(2)} s+\kappa_{(3)} s^3-3
\kappa_{(3)} s.
    \end{align*}
Moreover,
    \begin{align*}
p(y, \eta_{(2 3)}) = 2(a(r) + a(s)) (\kappa_{(1)}-1),
\quad
p(y, \eta_{(1 k)}) = 2(a(r) + a(s)) \kappa_{(k)}, \quad
k = 2,3.
    \end{align*}

For our purposes, it is enough to compute the leading order terms with respect to $s$, in the limit $s \to 0$, of the first two 1-form components of $\hat{Y}_{(1 2 3)}$.
The cubic terms
    \begin{align*}
[\hat{Y}_{(\pi(1))}^\alpha, [\hat{Y}_{(\pi(2)), \alpha}, \hat{Y}_{(\pi(3)), \beta}] ], \quad \beta =0,1,
    \end{align*}
are of order $s$. Indeed, if $\beta=1$ then the last factor vanishes, and if $\beta=0$ then the last factor is of order $s$.
Hence for $\beta =0,1$,
\begin{multline*}
  \hat{Y}_{(1 2 3),\beta}  = \frac{1}{2}   \sum_{\pi \in S_3} \left(
2 \eta_{(\pi(3)), \alpha} [\hat{Y}_{(\pi(1)\pi(2))}^\alpha,  \hat{Y}_{(\pi(3)), \beta}]   - \eta_{(\pi(3)), \beta} [\hat{Y}_{(\pi(1)\pi(2))}^\alpha,  \hat{Y}_{(\pi(3)), \alpha}   ]  \right.\\\left. +
2  \eta_{(\pi(2)\pi(3)), \alpha} [\hat{Y}_{(\pi(1))}^\alpha,  \hat{Y}_{(\pi(2)\pi(3)), \beta}]   -  \eta_{(\pi(2)\pi(3)), \beta} [\hat{Y}_{(\pi(1))}^\alpha,  \hat{Y}_{(\pi(2)\pi(3)), \alpha}   ]  \right) + \mathcal O(s).
\end{multline*}
It is in principle straightforward to express $\hat{Y}_{(1 2 3),\beta}$ in terms of $\tilde b_{(j)}$, analogously to (\ref{two_fold_expansion}). We do not reproduce here the details of this long computation, however, we have verified the below expression (\ref{Y123}) using a computer algebra system, and our code is available online \cite{CLOP2020_code}. There holds
    \begin{multline}\label{Y123}
\hat{Y}_{(1 2 3),0}
=
\hat{Y}_{(1 2 3),1}
=
-6 s^{-1} [\tilde b_{(1)}, [\tilde b_{(2)}, \tilde b_{(3)}]]
\\
+\left(6 s^{-1} + \frac{3 r}{1+a(r)}\right)[\tilde b_{(2)}, [\tilde b_{(1)}, \tilde b_{(3)}]]
\\
+\left(-6 s^{-1} + \frac{3 r}{1+a(r)}\right)[\tilde b_{(3)}, [\tilde b_{(1)}, \tilde b_{(2)}]] + \mathcal O(s).
    \end{multline}

The terms of order $s^{-1}$ cancel out due to the Jacobi identity.
Hence
    \begin{align*}
\lim_{s \to 0} \hat{Y}_{(1 2 3),\beta}
= \frac{3 r}{1+a(r)} \lim_{s \to 0}\left([\tilde b_{(2)}, [\tilde b_{(1)}, \tilde b_{(3)}]] + [\tilde b_{(3)}, [\tilde b_{(1)}, \tilde b_{(2)}]] \right), \quad \beta = 0,1.
    \end{align*}
Taking $b_{(3)} = b_{(2)}$ yields
    \begin{align*}
\frac{1+a(r)}{6 r} \lim_{s \to 0} \hat{Y}_{(1 2 3),\beta}
= \lim_{s \to 0} [\tilde b_{(2)}, [\tilde b_{(1)}, \tilde b_{(2)}]] =
\P_{y \gets x_{(1)}}^{A,\Ad} [b_{(2)}, [b_{(1)}, b_{(2)}]], \quad \beta = 0,1,
    \end{align*}
where we used the following simple consequence of the Jacobi identity
    \begin{align}\label{pt_product}
[\P_{y \gets x}^{A, \Ad} b_{(1)}, \P_{y \gets x}^{A, \Ad} b_{(2)}] =
\P_{y \gets x}^{A, \Ad} [b_{(1)}, b_{(2)}], \quad b_{(1)}, b_{(2)} \in \mathfrak g,\ x,y \in \R^{1+3}.
    \end{align}
Indeed, let $W_j$, $j=1,2$, be the solutions of (\ref{adjoint_pt}) with $V=V_j$.
Then the Jacobi identity implies
    \begin{align*}
\p_t [W_1, W_2] &= -[[\pair{A, \dot \gamma}, W_1], W_2] - [W_1, [\pair{A, \dot \gamma}, W_2]]
\\&= [W_2, [\pair{A, \dot \gamma}, W_1]] + [W_1, [W_2, \pair{A, \dot \gamma}]]
= -[\pair{A, \dot \gamma}, [W_1, W_2]].
    \end{align*}
Thus $[W_1, W_2]$ solves (\ref{adjoint_pt}) with $V = [V_1, V_2]$ and (\ref{pt_product}) follows.

We apply (\ref{eqn : parallel transport for homogeneous})
to obtain
    \begin{align}\label{symbol_Y123}
\alpha_{(0)}^{-1} \lim_{s \to 0} \left(c
\sigma[Y_{(1 2 3),\beta}](z,\eta) \right) = \P_{z \gets y}^{A,\Ad} \P_{y \gets x_{(1)}}^{A,\Ad} [b_{(2)}, [b_{(1)}, b_{(2)}]], \quad \beta = 0,1,
    \end{align}
where $c = c(s) = - (1+a(r))(6r\alpha)^{-1} |\kappa_{(1)}\kappa_{(2)}\kappa_{(2)}|^{1-q}$ and $\alpha_{(0)} \in \C \setminus 0$ is independent from $A$.

\subsection{Principal symbol in temporal gauge}
\label{sec_prinsymb_gauge}

To finish the proof of Proposition \ref{prop_muloc}, we show that
for $\beta=1,2,3$,
    \begin{align}\label{L_Y}
\sigma[\p_{\epsilon_{(1)}}\p_{\epsilon_{(2)}}\p_{\epsilon_{(3)}}L_\beta(0, J_{2}(0, s), 0)](z,\eta)
= -\frac{\eta_\beta}{\eta_0} \sigma[Y_{(1 2 3),0}](z,\eta) + \sigma[Y_{(1 2 3),\beta}](z,\eta).
    \end{align}
Indeed, Proposition \ref{prop_muloc} follows from (\ref{symbol_Y123}) and (\ref{L_Y}) with $\beta = 1$.

Recall that $L(0, J_{2}(\epsilon, s), 0)$ is defined by $\mathscr T(V)|_\mho$ where $V = W + A$ and $W$ is as in (\ref{the_sol}).
To simplify the notation, we write
    \begin{align*}
V_{(k)} = \frac{\partial V}{\partial \epsilon_{(k)}}\bigg|_{\epsilon = 0},
\quad
V_{(kl)} = \frac{\partial^2 V}{\partial \epsilon_{(k)}\partial \epsilon_{(l)}}\bigg|_{\epsilon = 0},
\quad
V_{(123)} = \frac{\partial^3 V}{\partial \epsilon_{(1)}\partial \epsilon_{(2)}\partial \epsilon_{(3)}}\bigg|_{\epsilon = 0}.
    \end{align*}
As $A$ is smooth, $\sigma[V_{(123)}](z,\eta) = \sigma[Y_{(1 2 3)}](z,\eta)$. It remains to study how the principal symbol $\sigma[V_{(123)}]$ transforms under passing to the temporal gauge with $\mathscr T$.

Let $\U = \U(\epsilon)$ be as in (\ref{temporal_U}) with $V = V(\epsilon)$, and write
    \begin{align*}
U_{(k)} = \frac{\partial \U}{\partial \epsilon_{(k)}}\bigg|_{\epsilon = 0},
\quad
U_{(kl)} = \frac{\partial^2 \U}{\partial \epsilon_{(k)}\partial \epsilon_{(l)}}\bigg|_{\epsilon = 0},
\quad
U_{(123)} = \frac{\partial^3 \U}{\partial \epsilon_{(1)}\partial \epsilon_{(2)}\partial \epsilon_{(3)}}\bigg|_{\epsilon = 0}.
    \end{align*}
Recall that we are using the notation $A = \tilde A$ where $\tilde A$ is as in Proposition \ref{prop_source_to_sol}. In particular, $A|_\mho$ is in temporal gauge.
This, together with $V|_{\epsilon=0} = A$, implies that
$\U|_{\epsilon = 0} = \id$ in $\mho$.

We will consider $V$ and $\U$ near the point $z \in \mho$.
Recall that $Y_{(k)}$ is singular only in $\Gamma_{(k)}$
and that $Y_{(kl)}$ is singular only in $\Gamma_{(k)} \cup \Gamma_{(l)}$.
Therefore $V_{(k)}$ and $V_{(kl)}$ are smooth near $z$.
Moreover, as $\WF(V_{(k)})$ and $\WF(V_{(kl)})$ are disjoint from the characteristic set $\{\tau = 0\}$ of $\p_t$, the ordinary differential equation in (\ref{temporal_U}) implies that also $U_{(k)}$ and $U_{(kl)}$ are smooth near $z$.

Writing
$$
T = \frac{\partial^3 \mathscr T(V)}{\partial \epsilon_{(1)}\partial \epsilon_{(2)}\partial \epsilon_{(3)}}\bigg|_{\epsilon = 0},
$$
and differentiating (\ref{temporal_U}) in $\epsilon_1$, $\epsilon_2$ and $\epsilon_3$ at $\epsilon = 0$ yields that
\begin{multline*}
T =   d U_{(123)} + U^{-1}_{(123)} A +  A U_{(123)} + V_{(123)} +
\\ \frac 1 2 \sum_{\pi \in S_3} \bigg(   U^{-1}_{(\pi(1) \pi(2))} V_{(\pi(3))}    + U^{-1}_{(\pi(1))} V_{(\pi(2) \pi(3))}    +   V_{(\pi(1) \pi(2))} U_{(\pi(3))}  +   V_{(\pi(1))} U_{(\pi(2) \pi(3))} \\ + U^{-1}_{(\pi(1) \pi(2))} d U_{(\pi(3))} + U^{-1}_{(\pi(1))} d U_{(\pi(2) \pi(3))}  + U^{-1}_{(\pi(1) \pi(2))} A U_{(\pi(3))}  + U^{-1}_{(\pi(1))} A U_{(\pi(2) \pi(3))} \bigg),
\end{multline*}
where $U_{(123)}$ solves
\begin{equation*}
 \p_t U_{(123)} =  - V_{(123),0}  - \frac 1 2 \sum_{\pi \in S_3} \bigg(V_{(\pi(1)\pi(2)),0} U_{(\pi(3))} + V_{(\pi(1)),0} U_{(\pi(2)\pi(3))}\bigg).
\end{equation*}
In addition, $\U^{-1} \U = \id$ implies
$$
U^{-1}_{(123)} + \frac 1 2 \sum_{\pi \in S_3} \bigg(U^{-1}_{(\pi(1)\pi(2))} U_{(\pi(3))} +  U^{-1}_{(\pi(1))} U_{(\pi(2)\pi(3))} \bigg) + U_{(123)} = 0.
$$
Therefore, modulo smooth terms, near $z$ there holds
    \begin{align}\label{dotV}
T  =     dU_{(123)}
- U_{(123)} A + A U_{(123)} + V_{(123)},
\quad \p_t U_{(123)} = -V_{(123),0}.
    \end{align}

Near $z$ it holds that $V_{(123)}$ is a conormal distribution associated to the future flowout of $N^* (K_{(1)} \cap K_{(2)} \cap K_{(3)}) \cap \Sigma$, cf. (\ref{def_flowout}).
We refer to Appendix C of \cite{CLOP} for a precise description of this flowout.
As the flowout is contained in the characteristic set $\Sigma$ of $\Box_A$, it is disjoint from the characteristic set $\{\tau = 0\}$ of $\p_t$. The second equation in (\ref{dotV}) implies that $U_{(123)}$ is a conormal distribution associated to the same flowout near $z$.

We write $\hat X = \sigma[X](z,\eta)$ where $X = T, V_{(123)}, U_{(123)}$.
Then taking principal symbols in (\ref{dotV}) gives for $\beta = 0,1,2,3$,
    \begin{align*}
\hat{T}_\beta = i \eta_\beta \hat{U}_{(123)} + \hat{V}_{(123),\beta},
\quad
i \eta_0 \hat{U}_{(123)} = -\hat{V}_{(123),0}.
    \end{align*}
Solving for $\hat{U}_{(123)}$ in the second equation and substituting in the first one yields (\ref{L_Y}).
This finishes the proof of Proposition \ref{prop_muloc}, and hence also Proposition \ref{prop_analytic_step} is proven.

\section{Lie algebras with trivial centre}
\label{sec_trivial_centre}

The material that follows is quite classical and can be found in many texbooks on Lie algebras.
We start by defining notations and recalling basic results following mainly the exposition from \cite[Chapter 7]{Hall}.

Let $\g$ be the Lie algebra of a compact connected Lie group of matrices $G$ and let $\gc$ be its complexification. An element $Z\in \gc$ can be uniquely written as $Z=X+iY$ for $X,Y\in \g$,
and we define $Z^*=-X+iY$.
Note that $Z^*$ is the usual conjugate transpose of $Z$ in the case $\g = \mathfrak{u}(n)$.
There is an inner product on $\gc$ that is real-valued on $\g$ and that satisfies, see \cite[Proposition 7.4]{Hall},
\[\langle\text{ad}_{Z}(X),Y\rangle=\langle X, \text{ad}_{Z^*}(Y)\rangle, \quad X,Y,Z \in \gc.\]
If $\t$ is a maximal commutative subalgebra of $\g$, then
\[\h=\t+i\t\]
is a Cartan subalgebra of $\gc$ and its dimension is called the rank of $\gc$.
The roots of $\gc$ relative to $\h$ are those elements $\alpha\in\h$ such that
there is $0\neq X\in \gc$ so that
    \begin{align}\label{eigenvector}
[H,X]=\langle \alpha,H\rangle X, \quad \text{for all $H \in \h$},
    \end{align}
where we use the convention that the inner product is linear in the second variable (and anti-linear in the first one).  We let $\Delta$ be the collection of roots. By \cite[Proposition 7.15]{Hall} each root $\alpha$ belongs to $i\t$ and that we can decompose $\gc$ as a direct sum
\[\gc=\h\oplus \bigoplus_{\alpha\in\Delta}\g_{\alpha}\]
where $\g_{\alpha}$ contains the eigenvectors associated to $\alpha$, that is, the vectors $X$ satisfying (\ref{eigenvector}). Moreover, see \cite[Proposition 7.18, Theorems 7.19 and 7.23]{Hall},
\begin{enumerate}
\item each $\g_{\alpha}$ is 1-dimensional;
\item if $X\in \g_{\alpha}$ with $\alpha\in\Delta$, then $X^*\in\g_{-\alpha}$;
\item if $\gc$ has trivial center, the roots span $\h$.
\end{enumerate}

We can in fact pick linearly independent elements $X_{\alpha}\in \g_{\alpha}$ , $Y_{\alpha}=X^*_{\alpha}\in \g_{-\alpha}$ and $H_{\alpha}\in\h$
such that $H_{\alpha}$ is a multiple of $\alpha$ and such that $[X_{\alpha},Y_{\alpha}]=H_{\alpha}$, $[H_{\alpha},X_{\alpha}]=2X_{\alpha}$ and $[H_{\alpha},Y_{\alpha}]=-2Y_{\alpha}$. This generates an $\mathfrak{sl}(2,\C)$-subalgebra inside
$\gc$ and implies that the elements
\[E^{\alpha}_{1}:=\frac{i}{2}H_{\alpha};\,\,\;E^{\alpha}_{2}=\frac{i}{2}(X_{\alpha}+Y_{\alpha});\,\,\;E^{\alpha}_{3}=\frac{i}{2}(Y_{\alpha}-X_{\alpha})\]
belong to $\g$ and span a Lie subalgebra isomorphic to $\mathfrak{su}(2)$, see \cite[Corollary 7.20]{Hall}. Note that the set $\{E_{\alpha}^{1}, E^{2}_{\alpha}, E^{3}_{\alpha}\}_{\alpha\in\Delta}$ spans $\g$ over the reals if $\g$ has trivial centre.
The commutation relations of Pauli matrices imply that $\mathfrak{su}(2)$ is spanned by the nested commutators $[X,[X,Y]]$ with $X, Y \in \su(2)$. Hence the discussion above
immediately implies:

\begin{proposition} Let $\g$ be the Lie algebra of a compact connected Lie group of matrices. Assume that $\g$ has trivial centre. Then $\g$ is the linear span of $[X,[X,Y]]$ for $X,Y\in\g$.
\label{prop_nestedcommutators}
\end{proposition}

\section{The case of general Lie group}
\label{sec_general_G}
 Suppose now $G$ is any compact connected Lie group. In what follows
it is convenient to express some previous notions in slightly more abstract form.
Let $\omega\in \Omega^{1}(G,\g)$ be the (left) Maurer-Cartan 1-form of $G$. Given $\U\in G^{0}(\mathbb{D},p)$ we express
the gauge equivalence between $A,B\in \Omega^{1}(M,\g)$ as
\begin{equation}
\U^*\omega+\text{Ad}_{\U^{-1}}(A)=B,
\label{eq:ge}
\end{equation}
where $\text{Ad}:G\to GL(\g)$ is the usual Adjoint representation. For matrix Lie groups $\omega=g^{-1}dg$
and $\text{Ad}_{g}(a)=gag^{-1}$ for $a\in\g$ and we recover the expression (\ref{gauge_equiv}) for the gauge equivalence between
$A$ and $B$ that we have used so far.

Suppose now that $p:\widetilde{G}\to G$ is a covering of $G$, then $p$ is a Lie group homomorphism and
$p^*\omega_{G}=\omega_{\widetilde{G}}$. Given $\U\in G^{0}(\mathbb{D},p)$, there is a unique
$\widetilde{\U}\in \widetilde{G}^{0}(\mathbb{D},p)$ such that $p\circ \widetilde{\U}=\U$.
This is because the domain of $\U$ is simply connected and we are fixing the value of $\U$ at $p$ to be the identity. We deduce that \eqref{eq:ge} holds if and only if the following equation holds
\[\widetilde{\U}^*\omega_{\widetilde{G}}+\text{Ad}_{\widetilde{\U}^{-1}}(A)=B.\]
In other words, $A$ and $B$ are gauge equivalent via a gauge in $G^{0}(\mathbb{D},p)$ if and only if they are gauge
equivalent via a gauge in $\widetilde{G}^{0}(\mathbb{D},p)$. The same observation applies for gauges defined
near $\partial^{-}\mathbb{D}$. One very useful consequence is that the data seta $\mathcal D_{A}$ does
not really depend on the group $G$ as long as it has Lie algebra $\g$.

We are going to use this set up as follows. Every compact connected Lie group $G$ admits a finite cover
of the form $\mathbb{T}^r\times G_{1}$, where $\mathbb{T}^r$ is an $r$-torus and $G_{1}$ is a compact Lie group
with finite centre \cite[Theorem 8.1, p. 233]{Broecker1985}. At the level of the Lie algebra this corresponds to an orthogonal splitting $\g=\mathfrak{z}\oplus \g_{1}$, where $\g_{1}$ is the Lie algebra of $G_{1}$ and it has no centre. Given $A\in \Omega^{1}(M,\g)$ we split uniquely
\[A=A_{Z}+A_{1}\in \mathfrak{z}\oplus\g_{1}.\]

Now we claim:

\begin{lemma} Let $A,B\in \Omega^{1}(M,\g)$. Then $\mathcal D_{A}=\mathcal D_{B}$ iff  $\mathcal D_{A_{Z}}=\mathcal D_{B_{Z}}$ and  $\mathcal D_{A_{1}}=\mathcal D_{B_{1}}$.
\label{lemma:split}
\end{lemma}

\begin{proof} Using that elements in the centre $\mathfrak{z}$ commute with everything, a quick calculation shows that given $V\in C^{3}(\mathbb{D};T^*\mathbb{D}\otimes\g)$ with $V=V_{Z}+V_{1}$ we can write the curvature of $V$ as
\[F_{V}=F_{V_{1}}+dV_{Z}\]
since $d_{V}=d_{V_{1}}$. Hence
\[d^{*}_{V}F_{V}=d_{V_{1}}^{*}(F_{V_{1}}+dV_{Z})=d^{*}_{V_{1}}F_{V_{1}}+d^{*}_{V_{1}}dV_{Z}.\]
Again using commutativity, $d_{V_{1}}^*dV_{Z}=d^*dV_{Z}$ since $dV_{Z}$ is also in the centre.
Hence
\begin{equation*}
d_{V}^*F_{V}=d^*dV_{Z}+d^{*}_{V_{1}}F_{V_{1}}\in \mathfrak{z}\oplus \g_{1}.
\end{equation*}
This implies that $d_{V}^*F_{V}=0$ in $\mathcal D\setminus\mho$ iff $d^*dV_{Z}=d^{*}_{V_{1}}F_{V_{1}}=0$
in $\mathcal D\setminus\mho$ and the lemma follows.
\end{proof}

We can deal with the abelian component $A_{Z}$ directly by unique continuation.

\begin{lemma}\label{lemma:uniquecontinuation}
If $\mathcal D_{A_{Z}}=\mathcal D_{B_{Z}}$, then there is $u\in C^{\infty}(\mathbb{D};\mathbb{T}^{r})$ with $u(p)=\text{\rm id}$ such that
\[B_{Z}=A_{Z}+u^{-1}du.\]
\end{lemma}
\begin{proof} It suffices to prove the claim for $r=1$, i.e. in the case of the circle $S^{1}$.
To avoid cluttering the notation we drop the subscript $``Z"$ during the proof.
If the group is abelian, the Yang--Mills equations reduces to the Maxwell equation $d^*F_{A}=0$, where $F_{A}=dA$. Since $dF_{A}=0$, the curvature satisfies $\Box F_{A}=0$, where $\Box=d^*d+dd^*$.
The gauges $u\in C^{\infty}(\mathbb{D};S^{1})$ all have the form $u=e^{i\phi}$ for $\phi$ a real-valued function since $\mathbb{D}$ is simply connected.

Since $A\in \mathcal D_{A}=\mathcal D_{B}$, there is $V$ with $d^*F_{V}=0$ in $\mathbb D\setminus\mho$, $V\sim B$ near $\p^- \mathbb D$ and $A|_{\mho}=V|_{\mho}$. Thus $d^*F_{V}=0$ in $\mathbb D$.
It follows that $\Box(F_{A}-F_{V})=0$ in $\mathbb{D}$ and $F_{A}=F_{V}$ in $\mho$ and by Holmgren's unique continuation principle, $F_{A}=F_{V}$ in $\mathbb{D}$, i.e. $d(A-V)=0$.
Since $\mathbb{D}$ is simply connected, $A$ and $V$ are gauge equivalent in $\mathbb{D}$.
But since $V\sim B$ near $\p^- \mathbb D$, it follows that $A$ and $B$ are gauge equivalent near  $\p^- \mathbb D$.
Proposition \ref{prop_abstract_uniq_pointed} implies now that $A$ and $B$ are gauge equivalent in the whole $\mathbb D$.
\end{proof}

We are now ready to prove our main result.

\begin{proof}[Proof of Theorem \ref{th_main}] We consider the finite cover $\mathbb{T}^r\times G_{1}$ of $G$ as above.
By Lemma~\ref{lemma:split} we know that $\mathcal D_{A_{Z}}=\mathcal D_{B_{Z}}$ and  $\mathcal D_{A_{1}}=\mathcal D_{B_{1}}$.  Let $u$ be the gauge from Lemma \ref{lemma:uniquecontinuation}.
We have already proven Theorem \ref{th_main}
in the case that $G = G_{1}$, since it has finite centre. Thus there is $\U\in G^{0}_{1}(\mathbb{D},p)$ so that $A_{1}$ and $B_{1}$ are gauge equivalent via $\U$. Finally, $p\circ (u,\U)\in G^{0}(\mathbb{D},p)$ gives a gauge equivalence between $A$ and $B$ as desired.
\end{proof}

\begin{appendix}

\section{Elementary computations}
\label{appendix_computations}

\subsection{The Hodge star operator on Minkowski space \texorpdfstring{$\mathbb{R}^{1+3}$}{}}

In this section we use the Cartesian coordinates $x^0, \dots,x^3$ on $\mathbb{R}^{1+3}$ and write $\pair{\cdot, \cdot}$ for the Minkowski metric with the signature $(- + + +)$.
We define also $\vol = dx^0 \wedge \dots \wedge dx^3$.

\begin{definition}\label{def_star}
 The Hodge star operator $\star$ for any forms $\omega$ and $\eta$
of the same degree is the linear map defined by
$\omega \wedge (\star \eta)
= \langle \omega, \eta \rangle  \vol$
where $\langle \omega, \eta \rangle = \det(\langle\omega_j, \eta_k\rangle)$ if $\omega = \omega_1 \wedge \cdots \wedge \omega_r$ and $\eta = \eta_1 \wedge \cdots \wedge \eta_r$ for some 1-forms $\omega_j$ and $\eta_j$.\end{definition}

In order to express the Yang--Mills equations and their linearizations in local coordinates, we will need the following lemma:
\begin{lemma}
Writing $g^{\alpha \beta} = \pair{dx^\alpha, dx^\beta}$
there holds
    \begin{align}\label{star1}
\star (dx^\alpha \wedge \star dx^\beta)
&= -g^{\alpha \beta},
\\\label{star2}
\star (dx^p \wedge \star (dx^\alpha \wedge dx^\beta) )
&=  g^{p\beta} dx^\alpha - g^{p\alpha} dx^\beta.
    \end{align}
In (\ref{star2}) it is assumed that $\alpha \ne \beta$.
\end{lemma}
\begin{proof}
Taking $\omega = \eta = \vol$ in Definition \ref{def_star}, we see that $\star \vol = g^{00} \cdots g^{33} = -1$. Then (\ref{star1}) follows immediately:
    \begin{align*}
\star (dx^\alpha \wedge \star dx^\beta) = g^{\alpha \beta} \star \vol = -g^{\alpha \beta}.
    \end{align*}

Let us turn to (\ref{star2}). Let $\alpha \ne \beta$
and choose indices $j$, $k$ and a sign $\epsilon = \pm 1$ so that
    \begin{align*}
dx^\alpha \wedge dx^\beta \wedge dx^j \wedge dx^k = \epsilon \vol.
    \end{align*}
Now $\star (dx^\alpha \wedge dx^\beta) = c dx^j \wedge dx^k$ for a sign $c=\pm 1$ that satisfies
    \begin{align*}
c \epsilon \vol = c (dx^\alpha \wedge dx^\beta )\wedge (dx^j \wedge dx^k)
= \eta \vol.
    \end{align*}
where
$\eta = \pair{dx^\alpha \wedge dx^\beta, dx^\alpha \wedge dx^\beta} = g^{\alpha \alpha} g^{\beta \beta}$.
Both sides of (\ref{star2}) vanish if $p \ne \alpha$ or $p \ne \beta$. Suppose now that $p=\alpha$, the case $p=\beta$ is analogous and we omit its proof. There holds
$\star (dx^\alpha \wedge dx^j \wedge dx^k) = c' dx^\beta$
for a sign $c'=\pm 1$ that satisfies
    \begin{align*}
c' \epsilon \vol =
c' dx^\alpha \wedge dx^\beta \wedge dx^j \wedge dx^k =
c' (dx^\alpha \wedge dx^j \wedge dx^k) \wedge dx^\beta
= \eta' \vol,
    \end{align*}
where $\eta' = g^{\alpha \alpha} g^{jj} g^{kk}$.
Solving for $c$ and $c'$ gives
    \begin{align*}
c c' = \epsilon^2 \eta \eta' = g^{\alpha \alpha} g^{\alpha \alpha} g^{\beta \beta} g^{jj} g^{kk} = - g^{\alpha \alpha}.
    \end{align*}
\end{proof}

\subsection{The adjoint \texorpdfstring{$d_A^*$}{of the covariant derivative} in coordinates}

Using the formulas (\ref{star1})--(\ref{star2})
we can easily find expressions for $d_A^* = \star d_A \star$ in the Cartesian coordinates.

\begin{lemma}
\label{lemma:aux2}
If $X= X_\alpha dx^\alpha$, then $$d_A^\ast X = - \left( \partial^\alpha X_\alpha + [A^\alpha, X_\alpha]\right).$$
If $Y= Y_{\alpha \beta} dx^\alpha \wedge dx^\beta$, then $$d_A^\ast Y =
	 \left( \partial^\alpha Y_{\beta\alpha} + [A^\alpha, Y_{\beta\alpha }]\right) dx^\beta
- \left( \partial^\alpha Y_{\alpha\beta} + [A^\alpha, Y_{\alpha \beta}]\right) dx^\beta.$$
\end{lemma}
\begin{proof}
We have
    \begin{align*}
d_A^\ast X = (\p_\alpha X_\beta + [A_\alpha, X_\beta])\star (dx^\alpha \wedge \star dx^\beta)
=  - \partial^\alpha X_\alpha - [A^\alpha, X_\alpha],
    \end{align*}
and
    \begin{align*}
d_A^\ast Y &=
(\p_p Y_{\alpha \beta} + [A_p, Y_{\alpha \beta}]) \star (dx^p \wedge \star (dx^\alpha \wedge dx^\beta) )
\\&
( \partial^\beta Y_{\alpha\beta} + [A^\beta, Y_{\alpha \beta}]) dx^\alpha
- ( \partial^\alpha Y_{\alpha\beta} + [A^\alpha, Y_{\alpha \beta}]) dx^\beta.
    \end{align*}
\end{proof}

\subsection{Proofs of (\ref{eqn : dastarliebracket})--(\ref{eqn : cubic term})}

In some of our computations we encounter
terms of the form $\star[X,\star Y]\in \Omega^{1}$ for $X\in \Omega^{1}$ and $Y\in \Omega^{2}$. The next elementary lemma computes this term explicitly.

\begin{lemma}\label{lemma:aux}
If $X=X_{\alpha}dx^{\alpha}$ and $Y=Y_{\alpha\beta}dx^{\alpha}\wedge dx^{\beta}$ then
\[\star[X,\star Y]=[X^{\alpha},Y_{\beta\alpha}]dx^{\beta}-[X^{\alpha},Y_{\alpha\beta}]dx^{\beta}.\]
\end{lemma}
\begin{proof} We have
    \begin{align*}
\star[X,\star Y]= [X_p, Y_{\alpha \beta}] \star (dx^p \wedge \star (dx^\alpha \wedge dx^\beta) )
=
[X^\beta, Y_{\alpha \beta}] dx^\alpha
-[X^\alpha, Y_{\alpha \beta}] dx^\beta.
    \end{align*}
\end{proof}

We are now ready to prove (\ref{eqn : wstardaw})
that expands $\star [X, \star d_A Z]$ for $X, Z \in \Omega^1$ in coordinates. Using Lemma \ref{lemma:aux} with $Y_{\alpha\beta} = \p_\alpha Z_\beta + [A_\alpha, Z_\beta]$, we obtain
\begin{align*}
\star [X, \star d_A Z]
&=
-[X^\alpha, \partial_\alpha Z_\beta + [A_\alpha, Z_\beta]] dx^\beta
+[X^\alpha, \partial_\beta Z_\alpha + [A_\beta, Z_\alpha]] dx^\beta.
\end{align*}

We apply Lemma \ref{lemma:aux} with $Y_{\alpha\beta}$ replaced by $[Y_\alpha, Z_\beta]$,
to establish (\ref{eqn : cubic term}), giving $\star [X, \star [Y, Z]]$ for $X, Y, Z \in \Omega^1$ in coordinates as follows,
\begin{align*}
\star [X, \star [Y, Z]]
&=- [X^\alpha,   [Y_\alpha, Z_\beta]] dx^\beta + [X^\alpha,   [Y_\beta, Z_\alpha]] dx^\beta.
\end{align*}

Proof of (\ref{eqn : dastarliebracket}), giving analogous expansion of $d_A^\ast [X, Z]$ for $X,Z \in \Omega^1$, is more involved. Let us consider first the terms in the $\beta$th component of
    \begin{align}\label{dastar_step1}
d_A^\ast [X, Z] + [X, d_A^* Z]-[d_A^* X, Z]
    \end{align}
that contain derivatives. Using Lemma \ref{lemma:aux2} these read
    \begin{align*}
\p^\alpha [X_\beta, Z_\alpha] - \p^\alpha [X_\alpha, Z_\beta]
-([X_\beta, \p^\alpha Z_\alpha]-[\p^\alpha X_\alpha, Z_\beta])
= [\p^\alpha X_\beta, Z_\alpha] - [X_\alpha, \p^\alpha Z_\beta].
    \end{align*}
Similarly, the terms in the $\beta$th component of (\ref{dastar_step1}) that do not contain derivatives are
    \begin{align*}
&[A^\alpha, [X_\beta, Z_\alpha]] - [A^\alpha, [X_\alpha, Z_\beta]] - ([X_\beta, [A^\alpha, Z_\alpha]] - [[A^\alpha, X_\alpha], Z_\beta])
\\&\qquad=
- [Z_\alpha, [A^\alpha, X_\beta]] + [X_\alpha, [Z_\beta, A^\alpha]].
    \end{align*}
We used here the Jacobi identity. Hence we obtain (\ref{eqn : dastarliebracket}), that is,
    \begin{align*}
d_A^\ast [X, Z] &= [d_A^* X, Z]-[X, d_A^* Z]
\\&\qquad+
\left([\p^\alpha X_\beta + [A^\alpha, X_\beta], Z_\alpha] - [X_\alpha, \p^\alpha Z_\beta + [A^\alpha, Z_\beta]]\right) dx^\beta.
    \end{align*}

\subsection{Yang--Mills equations in coordinates}
For the convenience of the reader we prove the following well-known lemma.

\begin{lemma}\label{lem_YM_coord}
If $A = A_\alpha dx^\alpha$ then the components of $d^*_A F_A$ are given by
    \begin{align*}
\partial^\alpha \partial_\beta A_\alpha
-\partial^\alpha \partial_\alpha A_\beta
- [\partial^\alpha  A_\alpha, A_\beta]
- 2 [A^\alpha, \partial_\alpha A_\beta]
+ [A^\alpha, \partial_\beta A_\alpha]
- [A^\alpha, [A_\alpha, A_\beta]].
    \end{align*}
\end{lemma}
\begin{proof}
We apply Lemma \ref{lemma:aux2} with $Y_{\alpha \beta} = \p_\alpha A_\beta + \frac 1 2 [A_\alpha, A_\beta]$, to see that the components of $d^*_A F_A$ are
    \begin{align*}
&\p^\alpha \p_\beta A_\alpha
+ \frac 1 2 \p^\alpha [A_\beta, A_\alpha]
+ [A^\alpha, \p_\beta A_\alpha]
+ \frac 1 2 [A^\alpha, [A_\beta, A_\alpha]]
\\
-&\p^\alpha \p_\alpha A_\beta
- \frac 1 2 \p^\alpha [A_\alpha, A_\beta]
- [A^\alpha, \p_\alpha A_\beta]
- \frac 1 2 [A^\alpha, [A_\alpha, A_\beta]],
    \end{align*}
and the claim follows after combining the terms with factors $1/2$, and using
    \begin{align*}
\p^\alpha [A_\alpha, A_\beta]
+[A^\alpha, \p_\alpha A_\beta]
=
[\p^\alpha A_\alpha, A_\beta]
+ 2[A^\alpha, \p_\alpha A_\beta].
    \end{align*}
\end{proof}

\section{Direct problem}
\label{appendix_energy}

\subsection{An energy estimate}

\def\mtext{}

We write again $(x^0, x^1, x^2, x^3) = (t,x) \in \R^{1+3}$ for the Cartesian coordinates, and recall the sign convention (\ref{def_Minkowski_Box}) for the wave operator $\Box$. We write also $\nabla u = (\p_{x^1} u, \p_{x^2} u, \p_{x^3} u)$ and denote by $\cdot$ the Euclidean inner product on $\R^3$.

Let $X_j$, $j=1,2$, be first order
and $Y_j$, $j=1,2$, zeroth order differential operators on $\R^{1+3}$.
Suppose, furthermore, that $X_2$ is of zeroth order with respect to $t$ variable. We will consider the system
    \begin{align}\label{model}
\Box v + X_1 v + X_2 u &= f_1,
\\\notag
\p_t u + Y_1 v + Y_2 u &=  f_2.
    \end{align}
Here $v$ and $u$ are allowed to take values on a Hermitian vector bundle, but we do not emphasize this in the notation.

We prove an energy estimate for (\ref{model}).
Write $B(r) = \{ x \in \R^3 : |x| < r\}$.
Let $R > 0$ and define $r(t) = R - t$.
Consider the following local energy
    \begin{align*}
E(t) = \frac 1 2 \int_{B(r(t))}
\mathcal E(t,x)\, dx,
\quad \mathcal E = |\p_t v|^2 + |\nabla v|^2 + |v|^2 + |\nabla u|^2 + |u|^2,
    \end{align*}
and the norm of the source
$$
F(t) = \int_{B(r(t))}
\mathcal F(t,x)\, dx,
\quad \mathcal F = |f_1|^2 + |f_2|^2 + |\nabla f_2|^2.
$$

\begin{lemma}
Let $T > 0$ and define the cut cone
    \begin{align*}
\mathcal C = \{(t,x) \in \mathbb R^{1+3} : |x|< R-t,\ 0 < t < T \}.
    \end{align*}
Suppose that $v,u \in C^2(\mathcal C)$ satisfy (\ref{model}) in $\mathcal C$.
Then for a constant $C > 0$ that depends only on the $L^\infty(\mathcal C)$-norm of
the coefficients of $X_j$ and $W^{1,\infty}(\mathcal C)$-norm of
the coefficients of $Y_j$, $j=1,2$,
    \begin{align}\label{energy_B}
E(t)
&\le
e^{C t} E(0) + C \int_0^t e^{C(t-s)} F(s) ds, \quad 0 < t < T.
    \end{align}
\end{lemma}
\begin{proof}
We differentiate the local energy
    \begin{align*}
\p_t E
&=
\int_{B(r(t))} \p_t^2 v \p_t v  + \nabla v \cdot \nabla \p_t v + v \p_t v + \nabla u \cdot \nabla \p_t u + u \p_t u\, dx
-\frac 1 2 \int_{\p B(r(t))} \mathcal E dx.
    \end{align*}
We write $z_1 = - X_1 v - X_2 u + v+ f_1$ and $z_2 = -Y_1 v - Y_2 u+f_2$,
apply integration by parts to the second term in the first integral, and use (\ref{model}) to obtain
    \begin{align*}
\p_t E
&=
\int_{B(r(t))} z_1 \p_t v + \nabla u \cdot \nabla z_2 + u z_2 \,dx
+ \int_{\p B(r(t))}  \p_\nu v \p_t v - \frac 1 2 \mathcal E\,dx.
    \end{align*}
We have $|z_j|^2 \le C(\mathcal E+  \mathcal F)$, $j=1,2$, and $|\nabla z_2|^2 \le C(\mathcal E+  \mathcal F)$,
where the constant $C > 0$ depends only on the $L^\infty(\mathcal C)$-norm of
the coefficients of $X_j$ and $W^{1,\infty}(\mathcal C)$-norm of
the coefficients of $Y_j$, $j=1,2$.
Moreover,
    \begin{align*}
2 |\p_\nu v \p_t v| \le |\nabla v|^2 + |\p_t v|^2 \le \mathcal E,
    \end{align*}
and we obtain
    \begin{align*}
\p_t E
&\le C (E+{\mtext F}).
    \end{align*}
Now we can use Gr\"onwall's inequality, or simply notice that
$$
e^{C t} \p_t (e^{-C t} E) \le CF,
$$
leading to the energy estimate (\ref{energy_B}).
\end{proof}

The energy estimate (\ref{energy_B}) implies the following two uniqueness results.

\begin{lemma}\label{lem_fsop}
Suppose that $v,u \in C^2(\mathbb D)$, that the coefficients
of $X_j$ are in $L^\infty(\mathbb D)$
and that the coefficients
of $Y_j$ are in $W^{1,\infty}(\mathbb D)$ for $j=1,2$.
If $(v,u)$ is a solution to (\ref{model}) with $f_1=0$ and $f_2=0$ and if $(v,u)$ vanishes near $\p^- \mathbb D$, then $(v,u)$ vanishes in $\mathbb D$.
\end{lemma}
\begin{proof}
As $(v,u)$ vanishes near $\p^- \mathbb D$, also the extension of $(v,u)$ by zero to the cone
    \begin{align*}
\{(t,x) \in \mathbb R^{1+3} : |x|< 1-t,\ t > -1 \},
    \end{align*}
solves (\ref{model}) with $f_1=0$ and $f_2=0$.
Therefore the energy estimate (\ref{energy_B}) implies that $(v,u)$ vanishes.
\end{proof}

\begin{proposition}\label{prop_abstract_uniq_pointed}
Let $A, B \in \Omega^{1}(\mathbb D;\mathfrak{g})$ solve (\ref{eq:1}) in $\mathbb D$.
Suppose that $A \sim B$ near $\p^- \mathbb D$.
Then $A \sim B$ in $\mathbb D$.
\end{proposition}
\begin{proof}
We write
$\tilde A = \mathscr T(A)$ and $\tilde B = \mathscr T(B)$, see (\ref{temporal_U}).
As $A \sim B$ near $\p^- \mathbb D$ also
$\tilde A \sim \tilde B$ there. That is, there is $\U \in G^0(\mathbb D,p)$ such that
    \begin{align*}
\tilde A = \U^{-1} d \U + \U^{-1} \tilde B \U, \quad \text{near $\p^- \mathbb D$}.
    \end{align*}
As both $\tilde A$ and $\tilde B$ are in the temporal gauge, $\U$ does not depend on time and we may define $V = \U^{-1} d \U + \U^{-1} \tilde B \U$ in the whole $\mathbb D$. Now both $\tilde A$ and $V$ satisfy the Yang--Mills equations in $\mathbb D$. They are also both in the temporal gauge and coincide near $\p^- \mathbb D$. Pseudolinearization in Section \ref{sec_pseudolin}, together with Lemma \ref{lem_fsop}, implies that $\tilde A = V$ in $\mathbb D$.
Therefore $\tilde A \sim \tilde B$ in $\mathbb D$ and hence also $A \sim B$ there.
\end{proof}

\subsection{Linearized Yang--Mills equations in relative Lorenz gauge}

A linearization of (\ref{YM_Lorenz_with_cc}) can be solved using the following lemma.
For notational convenience we translate the origin in time so that the initial conditions are posed on $t=0$.

\begin{lemma}
Let $T > 0$ and write $M = (0,T) \times \R^3$.
Let $A \in \Omega^1(M, \g)$ be as in Proposition \ref{prop_direct}.
Let $f_1 \in H^k(M; T^* M \otimes \mathfrak g)$
and $f_2 \in H^{k+1}(M; \mathfrak g)$. Then
    \begin{align}\label{lin_WJ0}
\begin{cases}
\Box_A \dot W + \star[\dot W, \star F_A] - \dot J_0dt = f_1, & t \geq 0,
\\
\p_t \dot J_0 + [A_0, \dot J_0] = f_2, & t \ge 0,
\\
\dot W = 0,\ \dot J_0 = 0, & t \le 0,
\end{cases}
    \end{align}
has a unique solution $(\dot W,\dot J_0)$ and the map $\mathcal S(f_1,f_2) = (\dot W,\dot J_0)$
is continuous
    \begin{align}\label{S_cont}
\mathcal S :
H^k(M; T^* M \otimes \mathfrak g) \times H^{k+1}(M; \mathfrak g)
\to H^{k+1}(M; M \otimes \mathfrak g \oplus \mathfrak g).
    \end{align}
\end{lemma}

The system (\ref{lin_WJ0}) is of the form (\ref{model})
with
$v = \dot W$ and $u= \dot J_0$, and the coefficients of $X_j$ and $Y_j$, $j=1,2$, depend only on the background connection $A$ and are smooth.
Using the energy estimate (\ref{energy_B}), it is straightforward to show that (\ref{lin_WJ0}) has a unique solution.
However, we give a short proof based on the fact that the second equation in (\ref{lin_WJ0}) is independent from $\dot W$.

\begin{proof}
Solving the second equation gives $\dot J_0 \in H^{k+1}(M; \mathfrak g)$. Then $\dot W$ can be solved from the linear wave equation
    \begin{align*}
\Box_A \dot W + \star[\dot W, \star F_A] = f_1 + \dot J_0dt,
    \end{align*}
where $f_1 + \dot J_0dt \in H^k(M; T^* M \otimes \mathfrak g)$.
\end{proof}

\subsection{Proof of Proposition \ref{prop_direct}}

To simplify the notation in the proof, we write $H^k(M)$ also for Sobolev spaces of vector valued functions.
As $k \ge 4$, the Sobolev embedding theorem implies that
both $H^k(M)$ and $H^{k+1}(M)$ are Banach algebras, and also that
$H^{k+1}(M)$ embeds in $C^2(M)$.

We define
    \begin{align*}
Pu &=
\begin{pmatrix}
\Box_A W + \star[W, \star F_A] - J_0dt
\\
\p_t J_0 + [A_0, J_0]
\end{pmatrix},
\\
\mathcal K(u, J') &=
\begin{pmatrix}
-\mathcal N(W) + J_j dx^j
\\
-[W_0, J_0] - \p^j J_j + [A^j, J_j] + [W^j, J_j]
\end{pmatrix},
    \end{align*}
where $u = (W, J_0)$, $J' = (J_1, J_2, J_3)$ and $j=1,2,3$.
Then (\ref{YM_Lorenz_with_cc}) is equivalent to
    \begin{align}\label{YM_Lorenz_with_cc2}
\begin{cases}
P u = \mathcal K(u, J'), & t \ge 0,
\\
u = 0, & t \le 0.
\end{cases}
    \end{align}

Consider the map $\Phi(u, J') = u - \mathcal S \mathcal K(u, J')$ where $\mathcal S$ is as in (\ref{S_cont}).
Observe that if $\Phi(u, J') = 0$ then $u = \mathcal S \mathcal K(u, J')$ solves (\ref{YM_Lorenz_with_cc2}).
Let us show that
    \begin{align}\label{Phi_cont}
\Phi : H^{k+1}(M) \times H^{k+2}(M)
\to H^{k+1}(M).
    \end{align}
We have $\mathcal N(W) \in H^k(M)$ since $W$, the first component of $u$, is in $H^{k+1}(M)$ and since $H^k(M)$ is a Banach algebra.
Therefore the first component of $\mathcal K(u, J')$
is in $H^k(M)$. Similarly, using the fact that $H^{k+1}(M)$ is a Banach algebra, we have that the second component of $\mathcal K(u, J')$ is in $H^{k+1}(M)$. The regularity (\ref{Phi_cont}) follows then from (\ref{S_cont}).

The map $\Phi$ is a third order polynomial, and therefore it is smooth. Moreover, $\mathcal K(u, 0)$ contains only monomials of order two and three, and it follows that $\p_u \Phi(0,0) = \id$. The implicit function theorem gives a neighbourhood $\mathcal H$ of the zero function in
$H^{k+2}(M)$ and a smooth map $J' \mapsto u$ from $\mathcal U$ to $H^{k+1}(M)$ such that $\Phi(u(J'), J') = 0$ for all $J' \in \mathcal H$.

\section{Generation of \texorpdfstring{$\su(n)$}{the Lie algebra of the special unitary group} using nested commutators}
\label{sec_sun}

We recall the definition of generalized Gell-Mann matrices.
Denote by $E_{jk}$ the matrix with $1$ in the $jk$-th entry and $0$ elsewhere.
The three types of generalized Gell-Mann matrices in $\C^{n \times n}$ are as follows
\begin{itemize}
\item[] symmetric type: for $1 \le j < k \le n$ let $S_{jk} = E_{jk} + E_{kj}$.
\item[] antisymmetric type: for $1 \le j < k \le n$ let $A_{jk} = -i E_{jk} + i E_{kj}$.
\item[] diagonal type: for $1 \le l \le n - 1$ let $D_l$ be the matrix with $1$ in the $jj$-th entry for $1 \le j \le l$, $-l$ in the $jj$-th entry with $j=l+1$, and $0$ elsewhere.
\end{itemize}
The diagonal type matrices $D_l$ are typically normalized by multiplying them with $\sqrt{\frac{2}{l (l+1)}}$ but this is irrelevant for our purposes. A basis of $\su(n)$ is given by the matrices $i S_{jk}$, $i A_{jk}$ and $i D_l$.

In the case $n=2$, we obtain the Pauli matrices
$$
  S_{12} =
    \begin{pmatrix}
      0&1\\
      1&0
    \end{pmatrix}, \quad
  A_{12} =
    \begin{pmatrix}
      0&-i\\
      i&0
    \end{pmatrix}, \quad
  D_1 =
    \begin{pmatrix}
      1&0\\
      0&-1
    \end{pmatrix}.
$$

We define the nested commutator
$$
c(A,B) = [A,[A,B]].
$$
\begin{lemma}\label{lem_span_sun}
$\su(n)$ with $n \ge 2$ is the linear span of the set
$$
\{ c(A,B) : A,B \in \su(n) \}.
$$
\end{lemma}

Before giving the general proof, let us consider the case of $\su(2)$.
A straightforward computation shows that
$$
S_{12} = 4 c(A_{12},S_{12}), \quad A_{12} = 4 c(S_{12},A_{12}),
\quad
D_1 = 4 c(S_{12},D_1).
$$
Therefore the lemma holds in the case $n=2$.

\begin{proof}
The computation in the case $n=2$ generalizes immediately to
$$
S_{jk} = 4 c(A_{jk},S_{jk}), \quad A_{jk} = 4 c(S_{jk},A_{jk}).
$$
Also $D_1 = 4 c(S_{12},D_1)$. We will show using an induction that $D_l$ can be expressed as a linear combination of the nested commutators.
Denote the upper left $m \times m$ block of a matrix $A$ by $A|^m$ and the lower right $m \times m$ block by $A|_m$.
Then
$$
A_{23}|^3 = \begin{pmatrix}
 0 & 0 & 0 \\
 0 & 0 & -i \\
 0 & i & 0 \\
\end{pmatrix}, \quad
D_1|^3 = \begin{pmatrix}
 1 & 0 & 0 \\
 0 & -1 & 0 \\
 0 & 0 & 0 \\
\end{pmatrix},
$$
and the rest of the entries of $A_{23}$ and $D_1$ are zero. Therefore
$$
c(A_{23},D_1)|^3 =
\begin{pmatrix}
 0 & 0 & 0 \\
 0 & -2 & 0 \\
 0 & 0 & 2 \\
\end{pmatrix},
$$
with the rest of the entries zero.
It follows that
$$
D_2 = D_1 - c(A_{23},D_1) = \frac 1 4 c(S_{12},D_1) - c(A_{23},D_1).
$$

Analogously,
$$
A_{l,l+1}|^{l+1}|_2 = \begin{pmatrix}
0 & -i \\
i & 0 \\
\end{pmatrix}, \quad
D_{l-1}|^{l+1}|_2 = \begin{pmatrix}
 -(l-1) & 0 \\
 0 & 0 \\
\end{pmatrix},
$$
and hence
$$
c(A_{l,l+1},D_{l-1})|^{l+1}|_2 =
2\begin{pmatrix}
-(l-1) & 0 \\
0 & l-1 \\
\end{pmatrix},
$$
with the rest of the entries zero. Therefore
$$
D_l = D_{l-1} - \frac{l}{2(l-1)} c(A_{l,l+1},D_{l-1}).
$$
If $D_{l-1}$ is a linear combination of the nested commutators, then so is $D_l$.
\end{proof}

\end{appendix}

\bibliographystyle{abbrv}
\bibliography{main}

\ifoptionfinal{}{
\listoftodos
}

\end{document}